\documentclass[a4paper, reqno, 12pt]{amsart}
\usepackage{tikz-cd}
\usepackage[all]{xy}
\usepackage{amssymb}
\usepackage{amsmath,amscd,color}
\usepackage[all]{xy}
\usepackage{mathrsfs}
\usepackage{tensor}
\usepackage{stackrel}
\usepackage{hyperref}
\usepackage{mathrsfs}
\usepackage{hyperref}
\usepackage{mathtools}
\usepackage[margin=1.2 in]{geometry}
\usepackage{leftindex}
\usepackage{enumitem}

\definecolor{tocolor}{rgb}{.1,.1,.5}
\definecolor{urlcolor}{rgb}{.2,.2,.6}
\definecolor{linkcolor}{rgb}{.1,.1,.6}
\definecolor{citecolor}{rgb}{.6,.2,.1}
\hypersetup{ colorlinks=true, urlcolor=urlcolor,
  linkcolor=linkcolor, citecolor=citecolor}
\definecolor{darkgreen}{rgb}{0.0, 0.5, 0.0}

\providecommand{\U}[1]{\protect\rule{.1in}{.1in}}
\newtheorem{theorem}{Theorem}[section]
\newtheorem*{theorem*}{Theorem}
\newtheorem*{claim*}{Claim}
\newtheorem{theoremM}{Theorem}

\newtheorem{corollary}[theorem]{Corollary}
\newtheorem{corollaryM}[theoremM]{Corollary}
\newtheorem{proposition}[theorem]{Proposition}
\newtheorem{propositionM}[theoremM]{Proposition}
\newtheorem*{proposition*}{Proposition}
\newtheorem{lemma}[theorem]{Lemma}

\theoremstyle{definition}
\newtheorem{definition}[theorem]{Definition}

\newtheorem{example}[theorem]{Example}
\newtheorem{remark}[theorem]{Remark}
\numberwithin{equation}{section}

\newcommand{\mf}[1]{\mathfrak{#1}}

\newcommand{\toto}{\rightrightarrows}

\newcommand{\LL}{\mathcal{L}}                   
\newcommand{\GG}{\mathcal{G}}                   
\renewcommand{\:}{\colon}                       
\renewcommand{\d}{{\mathrm{d}}}                   
\newcommand{\D}{{\mathrm{D}}}                   
\newcommand{\DD}{\mathcal{D}}                   
\newcommand{\FF}{\mathcal{F}}                   

\newcommand{\Hom}{\operatorname{Hom}}       
\newcommand{\End}{\operatorname{End}}       
\newcommand{\id}{\operatorname{id}}         
\newcommand{\im}{\operatorname{im}}         
\newcommand{\Der}{\operatorname{Der}}       
\newcommand{\UD}{\Breve{\mathrm{D}}}                        

\newcommand{\C}{\mathcal{C}}              
\newcommand{\HH}{\mathcal{H}}				

\DeclareMathOperator{\Diff}{Diff}           
\newcommand{\loc}{\mathrm{loc}}             
\newcommand{\Diffloc}{\Diff_{\loc}}         
\newcommand{\dom}{\operatorname{dom}}       
\DeclareMathOperator{\Bis}{Bis}             
\newcommand{\Bisloc}{\Bis_{\loc}}           
\DeclareMathOperator{\pr}{pr}               
\newcommand{\pullback}[2]{{}_{#1}\kern-\scriptspace{\times}_{#2}}                   

\newcommand{\onabla}{{\overline{\nabla}}}   

\DeclareMathOperator{\GL}{GL}               

\newcommand{\paction}{\mathbin{\hat{\curvearrowright}}}

\newcommand{\pltimes}{\hat{\ltimes}}

\begin{document}
\title{Relative algebroids and symmetries of Pfaffian fibrations}

\author{Wilmer Smilde}

\address{Department of Mathematics, University of Illinois at Urbana-Champaign, 1409 W. Green Street, Urbana, IL 61801 USA}
\email{wsmilde2@illinois.edu}

\date{\today}

\begin{abstract}
    Relative algebroids and Pfaffian fibrations are two frameworks recently developed to study geometric structures and PDEs with symmetries, but have structurally different foundations. In this article, we clarify the relation between the two. We show that every Pfaffian fibration canonically induces a relative algebroid, and that their prolongations and local solutions coincide. Moreover, we introduce two notions of symmetries of Pfaffian fibrations, namely internal symmetries and Pfaffian symmetries, and develop the theory for actions of Pfaffian groupoids by internal/Pfaffian symmetries. We show that such actions preserve the underlying relative algebroid in an appropriate sense. Our results apply in particular to partial differential equations with Lie pseudogroup symmetries.
\end{abstract}

\thanks{This work was partially supported by NSF grant DMS-2303586 and UIUC Campus Research Board Award RB25014.}

\maketitle

\setcounter{tocdepth}{1}
\tableofcontents

\section*{Introduction}

This article reconciles two modern approaches to PDEs and geometric structures with their symmetries, namely Pfaffian fibrations and groupoids, and relative algebroids.

The Pfaffian approach, developed by Accornero, Cattafi, Crainic, Yudilevich and Salazar in \cite{Accornero2021,AccorneroCattafi2022, AccorneroCattafiCrainicSalazarBook, Cattafi2020, CattafiCrainicSalazar2020, CrainicYudilevich2024, Salazar2013, Yudilevich2016}, studies the essential data of a PDE, encoding it in a Pfaffian fibration, as well as the essential data of a Lie pseudogroup: a Pfaffian groupoid. The focus of their work has been on geometric structures. A comprehensive treatment of the theory will be published in the upcoming book \cite{AccorneroCattafiCrainicSalazarBook}. The structural object of a Pfaffian fibration is a \emph{distribution}.

The study of relative algebroids, initiated by Yudelivich, Fernandes and the author of this article in \cite{FernandesSmilde2025} (see also \cite{Smilde2025Notes}), draws from aspects of Lie theory to study existence and classification problems for geometric structures \cite{Bryant2014}, and unifies the theory of Lie algebroids with that of formal PDEs. It explains the appearance of Lie algebroids behind several classification problems for geometric structures \cite{Bryant2001, FernandesStruchiner2014, FernandesStruchiner2019, FernandesStruchiner2021}, and codifies what is left of a PDE after taking the quotient by symmetries. The structural object of a relative algebroid is a \emph{relative derivation}.

Both approaches are seemingly different but build on the formal theory of PDEs, and many analogies can be observed. The present article establishes a robust relationship between the frameworks in two fundamental ways:
\begin{enumerate}
	\item Every Pfaffian fibration gives rise to a relative algebroid in a natural way.  The underlying relative algebroid shares the same formal theory as the Pfaffian fibration it is induced from.
	\item We introduce a framework for both pseudogroup and groupoid symmetries of Pfaffian fibrations, and show that they preserve the underlying relative algebroid in an appropriate sense.
\end{enumerate}
We believe that these results provide a clear picture of their interactions. 

The framework for symmetries of Pfaffian fibrations presented here is a generalization of PDEs with symmetries, but with some adaptations to make the interaction with relative algebroids more clear. More specifically:

\begin{enumerate}[label = (\roman*)]
	\item We introduce two different types of symmetries: \textbf{internal} and \textbf{Pfaffian symmetries}. Internal symmetries are a direct generalization of internal symmetries for PDEs \cite{AndersonKamranOlver1993}, while Pfaffian symmetries of a Pfaffian fibration interact well with the underlying relative algebroid.
	\item We show that internal symmetries \textit{prolong} to Pfaffian symmetries, and that conversely every Pfaffian symmetry of the prolongation is the prolongation of an internal symmetry when the Pfaffian fibration is \emph{involutive} (Theorem \ref{thm:InternalAndPfaffianSymmetries}).
	\item We develop a theory of groupoid actions of Pfaffian groupoids on Pfaffian fibrations by internal as well as by Pfaffian symmetries. Our notions more general than the \emph{Pfaffian actions} developed in \cite{Accornero2021, AccorneroCattafiCrainicSalazarBook, Cattafi2020, Salazar2013}, which allows for natural examples of PDEs with symmetries that are not Pfaffian actions (see Section \ref{sec:RemarksOnPfaffianActionsAndPDEsWithSymmetries}). 
	\item Pfaffian and internal actions are related as follows. The prolongation of an action by internal symmetries is a \emph{partial action} by Pfaffian symmetries. Partial groupoid actions have appeared in the literature before \cite{AnantharamanDelaroche2020, MarinPinedoRodrigue2025}, and seem to appear naturally in our context.
\end{enumerate}
The various notions of in this article are related as depicted by the following diagram.
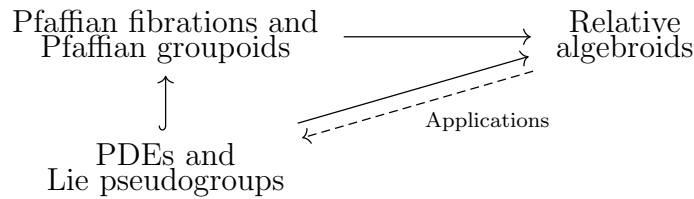
\[
	\begin{tikzcd}[column sep=6em]
	{\substack{ \mbox{ Pfaffian fibrations and } \\
				\mbox{ Pfaffian groupoids } }
			} \arrow[r] & \substack{ \mbox{ Relative }  \arrow[dl, shift left, dashed, "\text{Applications}"] \\ 
						\mbox{ algebroids } } \\
	{\substack{ \mbox{ PDEs and } \\
				\mbox{ Lie pseudogroups } }
	}	\arrow[u, hook] \arrow[ru, shift left] &				
	\end{tikzcd}
\]
The dashed arrow---indicating applications of the results of this article to PDEs with symmetries---will be a topic of further study. We will briefly indicate to the possible directions for applications at the end of this introduction. Furthermore, we believe that the connection between symmetries of Pfaffian fibrations and the current literature on symmetries of PDEs could be further developed.
\pagebreak

\subsection*{Overview of the article} 
In this paper, we work with \textbf{algebroids relative to a foliation}. Such an object $(B, \onabla, \D)$ consists of: 
\begin{itemize}[topsep=0pt]
	\item a manifold with foliation $(M, \FF)$, where $\FF\subset TM$ is an involutive subbundle,
	\item a vector bundle $(B, \onabla)\to (M, \FF)$ with flat $\FF$-connection $\onabla$,
	\item a derivation of degree 1 \textit{relative to $\FF$} :
	\[
		\D\: \Omega^\bullet_{(B, \onabla)} \to \Omega^{\bullet+1}_B.
	\]
\end{itemize}  
Here, $\Omega^\bullet_{B}$ stands for the sheaf of forms on $B$ (i.e. sections of $\wedge^\bullet B^*$), and $\Omega^\bullet_{(B, \onabla)}$ denotes the sheaf of flat forms (w.r.t. $\mathrel{\onabla}$). We refer to Appendix \ref{sec:RelativeAlgebroids}, or the articles \cite{FernandesSmilde2025, Smilde2025Notes} for more details on this notion.

\subsubsection*{From Pfaffian fibrations to relative algebroids} In Section \ref{sec:PfaffianFibrations}, we recall the basic definitions of Pfaffian fibrations and explain how every Pfaffian fibration induces a relative algebroid in a natural way. In this article, we mostly work from the perspective of distributions, so, to us, a \textbf{Pfaffian fibration} $(P, \C, \pi)$ is a submersion $\pi\: P\to X$ with a distribution $\C \subseteq TP$ that is transverse to $\ker T\pi$ and for which $\C^\pi:= \C \cap \ker T\pi$ is involutive. A (local) \textbf{solution} to a Pfaffian fibration is a (local) holonomic section of $\pi$, i.e. a (local) section $\sigma\: X\dasharrow P$ with image tangent to $\C$. 

A particularly important example is the bundle of first jets $\pi\: J^1q\to X$ of a submersion $q\: Q\to X$. The \textbf{Cartan distribution} $\C^1$ makes $(J^1q, \C^1, \pi)$ into a Pfaffian fibration. Locally, this is the universal example \cite[Proposition 3.23]{CattafiCrainicSalazar2020}. In this case, the holonomic sections are precisely those of the form $j^1\sigma$ for some section $\sigma$ of $q$.

A \textbf{partial differential equation (PDE)} is a subspace $\mathcal{E}\subset J^1q$ for some submersion $q: Q\to X$. By restricting the Cartan distribution $\C^1$ on $J^1q$ to $\mathcal{E}$ we obtain a Pfaffian fibration $(\mathcal{E}, \C^1, \pi)$ whose holonomic sections correspond precisely to the solutions of $\mathcal{E}$. 

A central result in \cite{FernandesSmilde2025} is that every PDE canonically induces a relative algebroid. If $\mathcal{E}\subset J^1q$ is a PDE on sections of some submersion $q\: Q\to X$, then the vector bundle of the relative algebroid underlying the PDE is $\pi^*TX\to J^1q$, and the derivation corresponds to the horizontal derivation in the corner of the variational bicomplex.

We extend this result to Pfaffian fibrations in Section \ref{sec:PfaffianFibrationsAsRelativeAlgebroids}. In particular, we show that to any Pfaffian fibration $(P, \C, \pi)$ one can naturally associate a relative algebroid $(\pi^*TX, \onabla, \D^\C)$, where the derivation is entirely determined by the distribution $\C\subset TP$ and the de Rham differential $\d$ on $TX$. This algebroid is not relative to the submersion $\pi$ but rather relative to the foliation $\C^\pi= \C \cap \ker T\pi$. The following theorem establishes a tight relationship between the formal theory of a Pfaffian fibration and of its underlying relative algebroid.

\begin{theoremM}[Theorem \ref{thm:PfaffianFibrationsAsRelativeAlgebroids}]
	Let $(P, \C, \pi)$ be a Pfaffian fibration with underlying relative algebroid $(\pi^*TX, \onabla, \D^\C)$. Then:
	\begin{enumerate}
		\item The prolongation spaces of $(P, \C, \pi)$ and of $(\pi^*TX, \onabla, \D^\C)$ are canonically isomorphic. In particular, $(P, \C, \pi)$ is 1-integrable iff $(\pi^*TX, \onabla, \D^\C)$ is.
		\item The relative algebroid underlying the prolongation of $(P, \C, \pi)$ is canonically isomorphic to the prolongation of $(\pi^*TX, \onabla, \D^\C)$.
		\item Germs of holonomic section of $(P, \C, \pi)$ are in one-to-one correspondence with germs of realizations of $(\pi^*TX, \onabla, \D^\C)$. 
	\end{enumerate}
\end{theoremM}

\subsubsection*{Pfaffian groupoids} In Section \ref{sec:PfaffianGroupoids}, we recall some general notions about Pfaffian groupoids. In this article, a Pfaffian groupoid $(\GG, \HH)\toto M$ is a Lie groupoid $\GG\toto M$ with a multiplicative distribution $\HH\subset T\GG$ that makes the source map $s\: \GG\to M$ into a Pfaffian fibration. Because a Pfaffian groupoid is in particular a Pfaffian fibration, we can talk about prolongations and $k$-integrability. 

\subsubsection*{Symmetries of Pfaffian fibrations}
In Section \ref{sec:SymmetriesOfPfaffianFibrations}, we turn to our second theme of symmetries of Pfaffian fibrations. The essential distributions of a Pfaffian fibration $(P, \C, \pi)$ are $\C$ and $\C^\pi = \C \cap \ker T\pi$. The pseudgroups of \textbf{internal symmetries} and \textbf{Pfaffian symmetries} are, respectively, 
\begin{align*}
	\Diffloc(P,\C) &= \{ \varphi\in \Diffloc(P) \ | \ T\varphi(\C)\subseteq \C \},\\
	\Diffloc(P, \C, \pi) &= \{ \varphi\in \Diffloc(P, \C) \ | \ T \varphi(\C^\pi) \subseteq \C^\pi \}.
\end{align*}
We show that:
\begin{enumerate}
	\item Pfaffian symmetries induce local isomorphisms of the underlying relative algebroid (Proposition \ref{prop:PfaffianSymmetriesInduceAlgebroidSymmetries}).
	\item Internal symmetries of $(P, \C, \pi)$ induce Pfaffian symmetries of the prolongation $(P^{(1)}, \C^1, \pi^{(1)})$. Moreover, when $(P, \C, \pi)$ is \emph{involutive}, Pfaffian symmetries of the prolongation are (locally) prolongations of internal symmetries (Theorem \ref{thm:InternalAndPfaffianSymmetries}).
\end{enumerate}

Perhaps more interesting are the Pfaffian groupoids of internal and Pfaffian symmetries, defined as subgroupoids of the groupoid first jets of diffeomorphisms:
\begin{align*}
	\GG_\C &= \{ j^1_x\varphi \in J^1\Diffloc(P) \ | \ T_x\varphi(\C_x) \subseteq \C_{\varphi(x)}\},\\
	\GG_{\C, \pi} &= \{ j^1_x\varphi\in J^1\Diffloc(P) \ | \  T_x\varphi(\C_x) \subseteq \C_{\varphi(x)} , \ T_x\varphi(\C^\pi_x) \subseteq \C^\pi_{\varphi(x)} \},
\end{align*}
and are equipped with the Cartan distribution $\HH^1$ inherited from $J^1\Diffloc(P)$. The central result of this part is the following:

\begin{theoremM}[Theorem \ref{thm:PfaffianDerivationInvariant}]
	Let $(P, \C, \pi)$ be a Pfaffian fibration with underlying relative algebroid $(\pi^*TX, \onabla, \D^\C)$. Then $\D^\C\in \Gamma(\DD^1_{\pi^*TX, \onabla)})$ is an invariant section for the canonical representation $\GG_{\C, \pi}^{(1)} \curvearrowright \DD^1_{(\pi^*TX, \onabla)}$.  
\end{theoremM}

\subsubsection*{Groupoid actions by Pfaffian symmetries} 
Often, a PDE comes with a specific pseudogroup symmetry given by the context (e.g. diffeomorphism-invariance). In the final part, Section \ref{sec:ActionsByPfaffianGroupoidsOnPfaffianFibrations}, we capture this by introducing a notion of an \textbf{action} of a Pfaffian groupoid on a Pfaffian fibration \textbf{by internal symmetries} and \textbf{by Pfaffian symmetries}. More precisely, a groupoid action 
\[
\xymatrix{
	(\GG, \HH) \ar@<0.5ex>[d] \ar@<-0.5ex>[d] & {(P, \C, \pi)} \ar@`{[l]+/l-0.2pc/+/d+1.8pc/,[l]+/l-0.2pc/+/u+1.8pc/}[0, 0]_{m} \ar[ld]^{\mu}    \\
	M & 
}
\]
of a Pfaffian groupoid $(\GG, \HH)\toto M$ on a Pfaffian fibration $(P, \C, \pi)$ is an \textbf{action by internal symmetries} when it satisfies
\[
	Tm\left( \HH \leftindex_{Ts}{\times}_{T\mu} \C \right) \subseteq \C,
\]
and an \textbf{action by Pfaffian symmetries} when in addition
\[
	Tm \left( \HH \leftindex_{Ts}{\times}_{T\mu} \C^\pi \right) \subseteq \C^\pi 
\]
is satisfied. These imply that the holonomic bisections of $(\GG, \HH)$ act by internal (resp. Pfaffian) symmetries of $(P, \C, \pi)$. 

As for the pseudogroups defined above, actions by internal symmetries and by Pfaffian symmetries are related through prolongation. However, if $(\GG, \HH)\curvearrowright (P, \C, \pi)$ is an action by internal symmetires, then it only induces a \textbf{partial action} (sometimes called a local action\footnote{For this article, we stick to the term \emph{partial action} and prefer to use local actions for actions of local Lie groupoids, which do not play a role in this article.}) on the prolongation. In Section \ref{sec:PartialActionsAndPfaffianGroupoids}, we briefly recall what a partial action of a groupoid is (see also \cite{AnantharamanDelaroche2020, MarinPinedoRodrigue2025}).

\begin{propositionM}[Corollary \ref{cor:ActionBySymmetriesAndProlongation}] Let $(\GG, \HH)\curvearrowright (P, \C, \pi)$ be an action by internal symmetries. Then:
	\begin{enumerate}
		\item The partial prolongation $J^1_\HH\GG$ has a natural \emph{partial action} by Pfaffian symmetries on the partial prolongation $J^1_\C P$.
		\item If $(\GG, \HH)$ is 2-integrable, and $(P, \C, \pi)$ is 1-integrable, then $\GG^{(1)}$ acts on $P^{(1)}$ by Pfaffian symmetries.
	\end{enumerate}
\end{propositionM}

Finally, if $(\GG, \HH)\toto M$ is sufficiently regular, we show that a Pfaffian action it leaves the derivation of the Pfaffian fibration invariant.

\begin{theoremM}[Theorem \ref{thm:ActionsByPfaffianSymmetriesInvariance}] Let $(\GG, \HH)\curvearrowright (P, \C, \pi)$ be an action by Pfaffian symmetries along a map $\mu\: P\to M$. Suppose that $(\GG, \HH)$ is 2-integrable. Then:
	\begin{enumerate}
		\item There is a natural representation of flat foliated bundles (Definition \ref{def:RepresentationOnFlatFoliatedVectorBundle}) $(\GG^{(2)}, \GG^{(1)}) \curvearrowright \pi^*TX$.
		\item The derivation $\D^\C\in \Gamma(\DD^1_{(\pi^*TX, \onabla)})$ is invariant for the natural representation $\GG^{(2)} \curvearrowright \DD^1_{(\pi^*TX, \onabla)}$. 
	\end{enumerate}
\end{theoremM}
\begin{corollaryM}[Corollary \ref{cor:DerivationPreservesInvariantSections}]
	If $(\GG, \HH)\curvearrowright (P, \C, \pi)$ is an action by Pfaffian symmetries and $\GG$ is 2-integrable, then the derivation $\D^\C$ preserves $\GG^{(1)}$-invariant sections:
	\[
		\D^\C\: \left(\Omega^1_{(\pi^*TX, \onabla)}\right)^{\GG^{(1)}}\to \left(\Omega^1_{\pi^*TX}\right)^{\GG^{(1)}}.
	\]
\end{corollaryM}
Often, the action of $\GG^{(1)}$ on $\pi^*TX$ descends to an action of $\GG$ itself on $\pi^*TX$. This happens for instance when the action is the prolongation of an action by internal symmetries. In that case, the corollary implies that the derivation preserves $\GG$-invariant sections, so it descends to the quotient! 

\subsubsection*{PDEs with symmetries and Pfaffian actions} In the last part, Section \ref{sec:RemarksOnPfaffianActionsAndPDEsWithSymmetries}, we demonstrate that the notion of a Pfaffian action, developed in \cite{AccorneroCattafi2022, AccorneroCattafiCrainicSalazarBook, Cattafi2020}, is not sufficient to capture PDEs with symmetries. Our example comes from the point symmetries of the first jet space $J^1q$ of a submersion $q\: Q\to X$. It gives rise to the following action
\[
\xymatrix{
	(J^1\Diffloc(q), \omega_Q) \ar@<0.5ex>[d] \ar@<-0.5ex>[d] & {(J^1q, \theta_Q, \pi)} \ar@`{[l]+/l-2.2pc/+/d+1.8pc/,[l]+/l-2.2pc/+/u+1.8pc/}[0, 0]_{m} \ar[ld]^{p_1}    \\
	Q,& 
}
\]
where $\Diffloc(q)$ is the pseudogroup of diffeomorphisms preserving $\ker Tq$, and $\omega_Q$ and $\theta_Q$ are the Cartan forms on $J^1\Diffloc(q)$ and $J^1q$, respectively. 

There is a groupoid map $\overline{q}\: J^1\Diffloc(q)\to J^1\Diffloc(X)$, and the action satisfies the following ``twisted Pfaffian" equation:
\[
	\pr^*_\GG\omega_Q - (g\cdot h) \circ \left(\overline{q}^*\omega_X\right) = m^*\theta_Q - g \cdot \left(\pr_{J^1q}^*\theta_Q\right),
\]
which shows that the action is not a Pfaffian action in the sense of \cite{AccorneroCattafiCrainicSalazarBook, Cattafi2020}. Instead, it is an action by Pfaffian symmetries. 
\subsubsection*{Appendices}
Appendix \ref{sec:RelativeAlgebroids} gives a summary of the main notions around relative derivations and relative algebroids, and states specific results used in the main part of the article. Some parts are new, such as the notions of a groupoid representation on flat foliated bundles and on relative algebroids (Sections \ref{sec:RepresentationOnFlatFoliatedBundles} and \ref{sec:RepresentationOnRelativeAlgebroids}).

Appendix \ref{app:TableauMapsAndInvolutivity} discusses involutivity, Cartan characters, and regular flags for tableau maps (introduced as generalized tableau in \cite{Salazar2013}), providing a technical result necessary for the proof of Theorem \ref{thm:InternalAndPfaffianSymmetries}.

\subsection*{Applications} A fruitful direction of applications will come from the study of the quotient relative algebroid of a PDE by symmetries. Let us first briefly mention how PDEs with pseudogroup symmetries fit into the framework of this paper.

For any pseudogroup $\Gamma$ (of diffeomorphisms) on a manifold $M$, the $k$-jets form a groupoid $J^k\Gamma$ whose multiplication is induced from composition of diffeomorphisms. If $J^k\Gamma\toto M$ is smooth, it becomes a Pfaffian groupoid $(J^k\Gamma, \HH^k)$ with the Cartan distribution $\HH^k$ inherited from $J^k\Diffloc(M)$. 

A PDE with symmetries $(E, \Gamma)$ is a PDE $E\subset J^k q$ on section of a submersion $q\: Q\to X$ together with a pseudogroup $\Gamma$ on $E$ that preserves the Cartan distribution. Interpreting $E$ as a Pfaffian fibration, $\Gamma$ is a pseudogroup of \emph{internal symmetries}. In this case there are partial actions $J^k \Gamma \hat{\curvearrowright} E^{(k)}$ on the prolongation spaces of $E$. Such actions are instances of \emph{actions by Pfaffian symmetries}. PDEs with symmetries have been extensively studied in the literature \cite{AndersonKamranOlver1993, AndersonIbragimov1979, KruglikovLychagin2008, KruglikovLychagin2016, OlverPohjanpelto2008, OlverPohjanpelto2009, Vinogradov2001}.

A huge advantage of the framework over relative algebroids is that it is stable under quotients. The current article sets it up by showing that the derivation of the relative algebroid underlying a Pfaffian fibration is invariant under symmetries. When quotient spaces exist, it descends to the structure of a relative algebroid on the quotient.

One instance of this phenomenon was already presented in \cite[Section 4.5]{Smilde2025Notes} to construct a relative algebroid on the space of $k$-jets of $n$-dimensional submanifolds $J^k(M, n)$ of a manifold $M$, and to construct the relative algebroid associated to a relative distribution. 

In a subsequent paper, we will study this quotient construction in detail and apply it to PDEs with symmetries. For instance, through a quotient construction, there should be a direct relationships between diffeomorphsim-invariant PDEs on Riemannian metrics and the relative algebroids found through Bryant's equations in \cite{Bryant2014}. 

The quotient of a PDE by symmetries lives over the space of invariants, and we expect there to be a close relationship between the quotient relative algebroid and the Lie-Tresse theorem \cite{KruglikovLychagin2016}.

A different direction is to establish the relationship between relative algebroids the moving-frame method for Lie pseudogroups developed by Olver and Pohjanpelto \cite{OlverPohjanpelto2009}. 

\subsection*{Acknowledgments} This article was written to properly answer a question raised by Luca Vitagliano following a talk by the author during the Mini-Workshop on Groupoids, Algebroids \& Differentiable Stacks, which took place at the University of Napoli Federico II, June 11–14, 2024.

The author especially thanks Luca Accornero for pointing out that actions by Pfaffian symmetries should be treated from the distributional sense (Remark \ref{rk:LucasRemark}), a foundational remark that inspired much of the content of this article. 

Additionally, the author thanks Luca Accornero, Francesco Cattafi, Marius Crainic and Maria Amelia Salazar for their answers to many questions about the framework of Pfaffian fibrations and Pfaffian actions, which will appear in full detail in their upcoming book \cite{AccorneroCattafiCrainicSalazarBook}. Access to an early draft has been particularly helpful.

\section{Pfaffian fibrations}\label{sec:PfaffianFibrations}
In this section, we establish the direct connection between Pfaffian fibrations and relative algebroids. A rapid introduction to relative algebroids has been included in Appendix \ref{sec:RelativeAlgebroids}, but for more details and examples we refer to \cite{FernandesSmilde2025, Smilde2025Notes}.
 
\subsection{Pfaffian fibrations} We start by summarizing the main aspects of the theory on Pfaffian fibrations---for a more detailed introduction that includes motivation and examples, see \cite{Cattafi2020, CattafiCrainicSalazar2020, Salazar2013}.

\begin{definition}[\cite{Salazar2013}]\label{def:PfaffianFibration}
    A \textbf{Pfaffian fibration} $(P, \C, \pi)$ consists of a submersion $\pi\: P\to X$ together with a distribution $\C\subset TP$ satisfying
    \begin{itemize}
        \item (transversality) $\C + \ker T\pi = TP$;
        \item ($\pi$-involutivity) $\C^\pi:= \C \cap \ker T\pi$ is an involutive distribution on $P$. 
    \end{itemize}
    A local section $\sigma$ of $\pi$ that is \textit{tangent to $\C$}, i.e. $\im T\sigma\subset \C$, is called \textbf{holonomic}.
\end{definition}

\begin{remark}
    It is often customary to work with a \textbf{Pfaffian form}, i.e. a one-form $\theta\in \Omega^1(P; E)$ with coefficients in a vector bundle $E\to P$ such that $\ker \theta = \C$ is a Pfaffian distribution. Any Pfaffian fibration can be given a Pfaffian form by setting $\theta_\C\: TP\to TP/\C$ to be the projection map. In that case, a local section $\sigma\in \Gamma_\pi$ is holonomic precisely when $\sigma^*\theta_\C = 0$. 
\end{remark}

\begin{example}[Cartan distribution on jet spaces and PDEs]\label{ex:CartanForm} 
    Let $q\: Q\to X$ be a submersion, and let $p_k\: J^kq \to J^{k-1}q$ and $q^{(k)}:J^k q\to M$ be the projections. The \textbf{Cartan distribution} $\C_k\in TJ^k q$ consists of those tangent vector $X\in T_{j^k_x\sigma} J^k q$ that are aligned image of $T_x (j^{k-1}\sigma)$. More precisely
    \begin{equation}\label{eq:CartanDistribution}
    	\C^k = \left\{ X\in \C_{j^k_x\sigma} \ | \ T_{j^k_x\sigma}p_k (X) = T_x(j^{k-1}\sigma)\left(T_{j^k_x\sigma} q^{(k)}(X)\right) \right\}.
    \end{equation}
    Alternatively, it is the kernel of the \textbf{Cartan form} $\theta \in \Omega^1(J^kq; p^*_k \ker Tq^{(k-1)})$ defined by
    \[
        \theta_{j^k_x \sigma}(X) = T_{j^k_x\sigma}p_k (X) - T_x(j^{k-1}\sigma)\left(T_{j^k_x\sigma} q^{(k)}(X)\right).
    \]
    Note that this is well-defined because $T_x(j^{k-1}\sigma)$ only depends on $j^k_x\sigma$. 
    
    The Cartan form detects whether a local section of $q^{(k)}$ is holonomic, namely, a local section $\tau$ of $q^{(k)}\: J^kq\to X$ is holonomic if and only if $\tau^*\theta = 0$. For the Cartan distribution, the holonomic sections are precisely those of the form $j^k\sigma$ for some section $\sigma$ of $q$.

   	 A \textbf{partial differential equation (PDE)} is a subbundle $\mathcal{E}\subset J^kq \to X$. Restricting the Cartan distribution to $\mathcal{E}$ makes any PDE into a Pfaffian fibration $(\mathcal{E}, \C^k, q^{(k)})$. Solutions to the PDE correspond to holonomic sections of this Pfaffian fibration.
\end{example}

Let $(P, \C, \pi)$ be a Pfaffian fibration. The \textbf{curvature map} of a Pfaffian fibration $(P, \C, \pi)$ is given by
\[
\kappa_\C\: \C \times \C \to TP/\C, \quad \kappa_\C(X, Y) = [\tilde{X}, \tilde{Y}]\mod \C.
\]
where $\tilde{X}, \tilde{Y}\in \Gamma(\C)$ are any sections through $X, Y\in \C$.

The \textbf{tableau map} or \textbf{symbol map} of a Pfaffian fibration $(P, \C, \pi)$ is
\[
	\tau_\C \: \C^\pi \to \Hom(\pi^*TX, TP/\C), \quad \tau_\C (v)(X) = \kappa_\C(v, \tilde{X})
\]
for $v\in \C^\pi$, $X\in TX$ and $\tilde{X}\in \C$ such that $T\pi(\tilde{X}) = X$. 

Recall that $J^1\pi$ can be identified with the bundle of (pointwise) splittings of $T\pi$:
\[
J^1\pi = \{ h\: T_{\pi(x)} X\to T_xP \  | \  T_x\pi \circ h = \id \}
\]
Any holonomic section $\sigma\: X\to P$ satisfies $\sigma^*\theta_\C = 0$ and consequently $\sigma^*\kappa_\C = 0$. The prolongation of a Pfaffian fibration takes into account these two differential consequences. 

\begin{definition}
    The \textbf{partial prolongation} of a Pfaffian fibration $(P, \C, \pi)$ consists of those splittings that are tangent to $\C$:
    \[
    	J^1_\C P = \{ h\in J^1\pi \ | \ h^*\theta_\C =0\}.
    \]
    The \textbf{prolongation} of $(P, \C, \pi)$ are those splittings that are tangent to $\C$ up to first order:
    \[
    	P^{(1)} = \{ h\in J^1\pi \ | \ h^*\theta_\C =0, \ h^*\kappa_\C = 0\}.
    \]
\end{definition}
The partial prolongation is always smooth and defines an affine bundle $p_1\: J^1_\C P\to P$ modeled on $\Hom(\pi^*TX, \C^\pi)$. The prolongation, however, may not be smooth. When it does, and $p_1\: P^{(1)}\to P$ is a submersion, we call $(P, \C, \pi)$ \textbf{1-integrable}. In that case, $p_1\: P^{(1)}\to P$ is an affine bundle modeled on
\[
	\tau_\C^{(1)} = \{ \xi\in \Hom(\pi^*TX, \C^\pi) \ | \ \delta_\tau(\xi) = 0\}
\]
where $\delta_\tau(\xi) (X, Y) = \tau(\xi(X))(Y) - \tau(\xi(Y))(X)$ for $X, Y\in \pi^*TX$.

Note that both the partial prolongation and the prolongations are PDEs, as they are subspaces of $J^1\pi$. This way, we can regard the (partial) prolongation as a new Pfaffian fibration $(P^{(1)}, \C^{1}, \pi^{(1)})$ (or $(J^1_\C P, \C^{1}, \pi^{(1)})$) with the Cartan distribution of a PDE (Example \ref{ex:CartanForm}).
    
The higher prolongations of $(P, \C, \pi)$ are defined iteratively:
\[
	(P, \C, \pi)^{(k)} := (P^{(k)}, \C^k, \pi^{(k)}) = (P^{(k-1)}, \C^{k-1}, \pi^{(k-1)})^{(1)}.
\]
When the first $k$ prolongations exist, then the Pfaffian fibration is called \textbf{$k$-integrable}. If it is $k$-integrable for all $k$, then it is \textbf{formally integrable.}

\begin{remark}\label{rk:reducedFibrationIntrinsicTorsion}
	It can happen that the projection $p_1\:P^{(1)}\to P$ is not surjective. For example, for a PDE this happens when the first order consequences have zeroth order implications. We denote the image by $\overline{P} = p_1(P^{(1)})$ and call it the \textbf{reduced Pfaffian fibration}. It coincides with the locus where the \textit{intrinsic torsion} of the Pfaffian fibration vanishes. We refer to \cite{AccorneroCattafiCrainicSalazarBook, CattafiCrainicSalazar2020} for what this means.
	
	If $P^{(1)}$ and $\overline{P}$ are manifolds, then $(\overline{P}, \C,\pi)$ is 1-integrable, even though $(P, \C, \pi)$ may not be.
\end{remark}

\subsection{Pfaffian fibrations as relative algebroids}\label{sec:PfaffianFibrationsAsRelativeAlgebroids}

Next we explain how a Pfaffian fibrations gives rise to a relative algebroid. We start by explaining how the various components arise from the Pfaffian fibration. The two essential distributions of a Pfaffian fibration $(P, \C, \pi)$ are the Cartan distribution $\C$ and the foliation $\C^\pi = \C \cap \ker T\pi$. The transversality condition implies that there is a canonical splitting
\[
    \nu(\C^\pi) \cong \C/\C^\pi \oplus \ker T\pi /\C^\pi
\]
for which the projection $\overline{T\pi}\: \nu(\C^\pi)\to \pi^*TX$ induces an isomorphism $\C/\C^\pi \cong \pi^*TX$. We denote the projections and inclusions associated to this splitting by the following:
\begin{align*}
    \Pi:= \overline{T\pi}&\: \nu(\C^\pi) \to \pi^*TX, & I&\: \pi^*T^*X \hookrightarrow \nu^*(\C^\pi),\\
    I&\:\pi^*TX\to \nu(\C^\pi), & \quad \Pi&\: \nu^*(\C^\pi) \to \pi^*T^*X.
\end{align*}

\subsubsection*{The flat foliated bundle} The bundle of the relative algebroid is $B\:= \pi^*TX\to P$. It has a canonical flat $\ker T\pi$-connection $\onabla$ that can be restricted to $\C^\pi \subset \ker T\pi$ to obtain a canonical flat $\C^\pi$-bundle $(B, \onabla)$. In other words, the connection of the bundle $(B, \onabla)$ is characterized by 
\[
    \onabla_v \pi^*X = 0, \quad \mbox{for $v\in \Gamma(\C^\pi)$ and $X\in \mf{X}(X)$},
\]
and extended through the Leibniz rule. 

There is a more concrete description of $\onabla$ in terms of the canonical flat $\C^\pi$-bundle $(\nu(\C^\pi), \onabla^{\mathrm{Bott}})$. While $\onabla^{\mathrm{Bott}}$ does not preserve $\Gamma(\C)$, we claim that
\begin{equation}\label{eq:FlatConnectionPfaffian}
    \onabla_v b = \Pi \left(\onabla^{\mathrm{Bott}}_v I(b) \right) 
\end{equation}
for $b\in \Gamma(B)$. Indeed, both sides vanish when $b = \pi^*X$ for some $X\in \mf{X}(X)$, and satisfy the same Leibniz rule.

For future reference we will also derive an explicit formula for $\onabla$ on $B^*$. First, we check that the Bott connection $\onabla^{\mathrm{Bott}}$ on $\nu^*(\C^\pi)$ is given by
\[
    \onabla^{\mathrm{Bott}}_v \alpha = -\iota_v \d \alpha
\]
for $v\in \Gamma(\C^\pi)$ and  $\alpha\in \Gamma(\nu^*(\C^\pi))$. By regarding $\nu^*(\C^\pi) \subseteq T^*P$ as the annihilator of $\C^\pi$, the expression $\d \alpha$ makes sense. Dualizing Equation (\ref{eq:FlatConnectionPfaffian}) yields:
\begin{equation}\label{eq:DualFlatConnectionPfaffian}
    \onabla_v \alpha = - \Pi(\iota_v \d I(\alpha))
\end{equation}
for $v\in \Gamma(\C^\pi)$, $\alpha \in \Gamma(B^*)$

\subsubsection*{The canonical relative derivation} The relative derivation $\D^\C$ induced by the Pfaffian fibration on $(B, \onabla)$ is characterized by the symbol $I\: B\to \nu(\C^\pi)$ and the equation $\D^\C \pi^*\alpha = \pi^*\d \alpha$ for $\alpha\in\Omega^\bullet_{TX}$. It can be seen as an extension of the de Rham differential on $TX$ to $(\pi^*TX, \onabla)$, a statement that will be made precise in Proposition \ref{prop:CharacterizationPfaffianFibrationsByDerivations}. Alternatively, the derivation $\D^\C$ can be defined through the equation 
\begin{equation}\label{eq:PfaffianDerivationExplicit}
    \D^\C \alpha = \Pi (\d I(\alpha)), 
\end{equation}
for $\alpha\in \Omega^\bullet_{(\pi^*TX, \onabla)}$. Here, $I(\alpha)$ is regarded as a section of $\nu^*(\C^\pi)\subset T^*P$ so that the exterior derivative makes sense. In Equation (\ref{eq:PfaffianDerivationExplicit}), it is crucial that $\alpha$ is a flat section of $(\wedge^\bullet B^*, \onabla)$, otherwise $\d I(\alpha)$ would not be a section of $\wedge^{\bullet+1}\nu^*(\C^\pi)$.

The bracket $[\cdot, \cdot]^\C$ dual to the derivation $\D^\C$ has a similar description. It is determined by
\[
    [\pi^*X, \pi^*Y]^\C = \pi^*[X, Y], \quad X, Y\in \mf{X}(X)
\]
with anchor $I\: B\to \nu(\C^\pi)$. Alternatively, the bracket is given by
\[
    [b_1, b_2]^\C = \Pi(\overline{[X_1, X_2]}),
\]
where $X_1, X_2\in \Gamma_{\C}$ are such that $\Pi(X_i) = b_i$ for two \emph{flat} section $b_1, b_2\in \Gamma_{(B, \onabla)}$. The bar above the Lie bracket indicates the projection to $\nu(\C^\pi)$. We have proven the following statement.

\begin{proposition}\label{prop:RelativeAlgebroidUnderlyingPfaffianFibration}
    A Pfaffian fibration $(P, \C, \pi)$ gives rise to an algebroid $(B, \overline{\nabla}, \D^\C)$ over $P$ relative to $\C^\pi$.
\end{proposition}

Next, we make precise how the Pfaffian fibration and the underlying relative algebroid relate. First, we need some preliminary results that characterize Pfaffian structures as derivations. 

Let $\pi\: P\to X$ be a submersion and $\FF\subset \ker T\pi$ a foliation. By restricting the canonical $\ker T\pi$-connection on $\pi^*TX$ to $\FF$, we obtain a foliated flat bundle $(\pi^* TX, \onabla)\to (P, \FF)$. Because $\FF$ is tangent to the fibers of $\pi$, there is a natural map $\pi_*\:\DD^1_{(TX, \onabla)}\to \pi^*\DD^1_{TX}$ obtained by restricting a derivation $\D\: \Omega^\bullet_{(\pi^*TX, \onabla)} \to \Omega^{\bullet + 1}_{\pi^*TX}$ to $\pi^*\Omega^\bullet_{TX}\subset\Omega^\bullet_{(\pi^*TX, \onabla)}$. Also, recall that we can interpret the de Rham differential $\d$ as a section of $\DD^1_{TX}$. 

\begin{proposition}\label{prop:CharacterizationPfaffianFibrationsByDerivations}
    Let $\pi\: P \to X$ be a submersion and $\FF\subset \ker T\pi$ a sub-foliation. There is a one-to-one correspondence between Pfaffian distributions $\C\subset TP$ with $\C\cap \ker T \pi = \FF$ and derivations $\D\in \Gamma(\DD^1_{(\pi^*TX, \onabla)})$ satisfying $\pi_* \D = \pi^*\d$.
\end{proposition}
\begin{proof}
    First, there is a one-to-one correspondence between Pfaffian distributions $\C\subset TP$ with $\C\cap \ker T\pi = \FF$ and \textbf{connections relative to $\FF$}, i.e. subbundles $\tilde{\C}\subset \nu(\FF)$ such that $\tilde{\C} \oplus \ker T\pi/\FF = \nu (\FF)$.

    Next, we establish a one-to-one correspondence between \textit{connections relative to $\FF$} and derivations $\D\in \Gamma(\DD^1_{(\pi^*TX, \onabla)})$ with $\pi_*\D = \pi^*\d$. One direction is through the proof of Proposition \ref{prop:RelativeAlgebroidUnderlyingPfaffianFibration}. For the converse, suppose we are given such a derivation $\D$. The anchor $\rho_\D\: \pi^*TX \to \nu(\FF)$ is determined by
    \begin{equation}\label{eq:AnchorDerivationRelativeConnection}
        \langle \d f, \rho_\D(X)\rangle = \langle \D f, X\rangle 
    \end{equation}
    for any $X\in TX$ and $f\in C^\infty_\FF$. Note that this is well-defined because $\d f$ annihilates $\FF$. If we take $f\in \pi^*C^\infty_X$, then it follows from $\pi_* \D = \pi^*\d f$ that $\rho_\D$ must be injective and that the map $T\pi\circ \rho_\D \: \pi^*TX\to \pi^*TX$ must be surjective. Therefore, the image of $\rho_\DD$ must be a complement of $\ker T\pi/\FF$ inside $\nu(\FF)$. In other words, $\rho_\D (\pi^*TX)$ is a connection relative to $\FF$. 
\end{proof}

\begin{theorem}\label{thm:PfaffianFibrationsAsRelativeAlgebroids}
    Let $(P, \C, \pi)$ be a Pfaffian fibration with underlying relative algebroid $(\pi^*TX, \onabla, \D^\C)$. 
    \begin{enumerate}
        \item The (partial) prolongation spaces of $(P, \C, \pi)$ and $(\pi^*TX, \onabla, \D^\C)$ are isomorphic. Consequently, $(P, \C, \pi)$ is 1-integrable if and only if $(\pi^*TX, \onabla, \D^\C)$ is.
        \item The tableau maps of $(P, \C, \pi)$ and $(\pi^*TX, \onabla, \D^\C)$ are canonically identified.
        \item If $(P, \C, \pi)$ is 1-integrable, then the relative algebroid $((\pi^{(1)})^*TX, \onabla, \D^{\C^1})$ underlying the prolongation $(P^{(1)}, \C^1, \pi^{(1)})$ is canonically isomorphic to the prolongation of $(\pi^*TX, \onabla, \D^\C)$ as a relative algebroid. 
        \item Germs of holonomic sections of $(P, \C, \pi)$ are in one-to-one correspondence with germs of realizations of $(\pi^*TX, \onabla, \D^\C)$ modulo diffeomorphism. 
    \end{enumerate}
\end{theorem}

\begin{proof}
    Recall from Corollary 5.5 in \cite{FernandesSmilde2025} that
    \[
        J^1\pi \cong \pi_*^{-1}(\d) = \{ \D_x\in \DD^1_{\pi^*TX} \ \big\vert \ \pi_*(\DD_x) = \d_{\pi(x)} \} 
    \]
    where the isomorphism is given by the symbol map $\sigma\: \DD^1_{\pi^*TX}\to \Hom(\pi^*TX, TP)$. 
    
    The equation $h_x^*\theta_\C = 0$ is equivalent to $h_x$ defining a \textit{pointwise extension}
    \[
        \D^{h_x}\: \Omega^\bullet_{\pi^*TX}\to \wedge^{\bullet+1}T_{\pi(x)}^*X
    \]
    of $\D^\C$ at $x\in P$. If $h_x$ satisfies $h^*_x\theta_\C=0$, then the equation $h^*_x\kappa_\C$ is equivalent to this extension being a completion, i.e. satisfying $\D^{h_x} \circ \D^\C =0$. It follows that $(P, \C, \pi)$ and $(\pi^*TX, \onabla, \D^\C)$ have the same (partial) prolongation spaces, which proves (1).
    
    For (2), the symbol map and the projection $\theta_\C: \nu(\C^\pi)\to TP/\C$ assemble into a map
    \[
    	\theta_\C \circ \sigma \: \DD^1_{(\pi^*TX, \onabla)}\to \Hom(\pi^*TX, TP/\C).
    \]
    The tableau map $\tau_{\D^\C}$ of $(\pi^*TX, \onabla, \D^\C)$ is related to the tableau map $\tau_\C$ of $(P, \C, \pi)$ through
    \[
    	\tau_\C = \theta_\C \circ \sigma \circ \tau_{\D^\C}.
    \]
    To prove this claim, let $v\in\Gamma(\C^\pi)$ and $X\in \mf{X}(X)$. Recall that the symbol of $\D^\C$ is the inclusion $I\: \C\hookrightarrow \nu(\C^\pi)$. Then
    \begin{align*}
    	\theta_\C \circ \sigma\left(\tau_{\D^\C}(v)\right)(\pi^*X) &= \theta_\C \circ \sigma(\onabla_v \D^\C)(\pi^*X) = \theta_\C\left( \onabla_v I(\pi^*X)\right)\\
    	&= \theta_\C\left([v, \tilde{X}]\right) = \tau_\C(v)(\pi^*X),
    \end{align*}
	where $\tilde{X}\in \Gamma(\C)$ is $\pi$-related to $X$.

    Item (3) follows because on one hand $(P^{(1)}, \C^1, \pi^{(1)})$ is the Pfaffian fibration corresponding to the PDE $P^{(1)}\subset J^1\pi$, whereas on the other hand, the prolongation of $(\pi^*TX, \onabla, \D^\C)$ is the relative algebroid underlying the PDE $P^{(1)} \subseteq J^1\pi$. 

    For item (4), let $\sigma\: U\to P$ be a local holonomic section of $(P, \C, \pi)$. We can then construct a realization $\Theta\: TU\to \pi^*TX$  of $(\pi^*TX, \onabla, \D^\C)$ covering $\sigma\: U\to P$ by setting $\Theta(X_x)= X_x\in (\pi^*TX)_{\sigma(x)}$. Conversely, if $\Theta\: TU\to \pi^*TX$ is a realization of covering some map $r\: U\to P$, then $T\pi\circ Tr$ is pointwise an isomorphism, so in a neighborhood of any point, there is a section $\sigma\: U \dashrightarrow P$ such that (in a neighborhood) $r = \sigma \circ (\pi \circ r)$. This section is holonomic because the image of $r$ is tangent to $\C$. So, up to (local) diffeomorphism, the realization determines a holonomic section.
\end{proof}

\subsection{Morphisms of Pfaffian fibrations}

Next, we discuss the naturality of the construction of the relative algebroid underlying a Pfaffian fibration.

\begin{definition}\label{def:morphismOfPfaffianFibration}
    A \textbf{morphism of Pfaffian fibrations} from $(P_1, \C_1, \pi_1)$ to $(P_2, \C_2, \pi_2)$ is a map $(\Phi, \varphi)\: P_1\to P_2$ covering $\varphi\: X_1\to X_2$ (that is, $\pi_2\circ \Phi = \varphi\circ \pi_1$) satisfying  $T\Phi(\C_1)\subseteq \C_2$. 
\end{definition}

\begin{remark}
    A morphism of Pfaffian fibrations $\Phi\: (P_1, \C_1, \pi_1)\to (P_2, \C_2, \pi_2)$ satisfies
    \[
        \Phi^*\theta_{\C_2} = \Phi_*\circ \theta_{\C_1},
    \]
    where $\Phi_*\: TP_1/\C_1\to TP_2/\C_2$. Note that such a notion of morphisms of Pfaffian fibrations was considered in \cite{CattafiCrainicSalazar2020} when the map $\Phi$ covers the identity on the base of the fibration. 
\end{remark}

The next proposition shows that the algebroid underlying the Pfaffian fibration is natural.

\begin{proposition}\label{prop:NaturalityOfPfaffianRelativeAlgebroid}
    Let $(\Phi, \varphi)\: (P_1, \C_1, \pi_1)\to (P_2, \C_2, \pi_2)$ be a morphism of Pfaffian fibrations. The canonical bundle map $\Phi_*\: \pi_1^*TX_1\to \pi_2^*TX_2$ is a morphism of relative algebroids.
\end{proposition}

\begin{proof}
    Since $\Phi$ covers $\varphi$, the bundle map $\Phi_*\: \pi_1^*TX_1\to \pi_2^*TX_2$ is $\Phi$ on the base and $T\varphi$ on the fibers.
    
    By definition, $T\Phi(\C_1) \subseteq \C_1$ and $T\Phi(\ker T\pi_1) \subseteq \ker T\pi_2$, which implies that $T\Phi(\C_1^{\pi_1})\subseteq \C_1^{\pi_2}$. In particular, the differential $T\Phi$ descends to $\overline{T\Phi}\: \nu(\C_1^{\pi_1})\to \nu(\C_2^{\pi_2})$ and preserves the splittings
    \[
    	\overline{T\Phi}\:\nu(\C_1^{\pi_1})\cong \C_1/\C_1^{\pi_1} \oplus \ker T\pi_1/\C_1^{\pi_1} \to \C_2/\C_2^{\pi_2}\oplus \ker T\pi_2/\C_2^{\pi_2}\cong  \nu(\C_2^{\pi_2}) 
    \]
    Denoting the dual of $\Phi_*$ by $\Phi^*:=(\Phi_*)^*$, both maps intertwine the projections $\Pi_1, \Pi_2$ and inclusion $I_1, I_2$ introduced in Section \ref{sec:PfaffianFibrationsAsRelativeAlgebroids}:
    \begin{align*}
		\Phi_*\circ \Pi_1 &= \Pi_2 \circ \overline{T\Phi}, & \overline{T\Phi}\circ I_1 &= I_2\circ \Phi_* \\
		I_1\circ \Phi^* &= \Phi^*\circ I_2, & \Pi_1\circ \Phi^* &= \Phi^*\circ \Pi_2.
    \end{align*}
    
    First, we have to show that $\Phi_*$ is a map of foliated vector bundles. According to Equation (\ref{eq:MapOfFoliatedVectorBundles}), the morphism condition is checked through the dual: for $\alpha\in \Gamma(\pi_2^*T^*X_2)$ and $v\in \C_1^{\pi_1}$, using Equation (\ref{eq:DualFlatConnectionPfaffian}):
    \begin{align*}
    	\onabla^1_v \Phi^*\alpha &= \Pi_1\left(\iota_v \d I_1 \Phi^*\alpha\right) = \Pi_1\left(\iota_v\Phi^*\d I_2\alpha\right)\\
    		&= \Pi_1\left(\Phi^*\iota_{\Phi_*(v)} \d I_2 \alpha \right) = \Phi^* \Pi_2\left( \iota_{\Phi_*(v)} \d I_2 \alpha \right) = \Phi^* \onabla^2_{\Phi_*(v)} \alpha,
    \end{align*}
    so $\Phi_*$ is indeed a morphism of flat foliated vector bundles.
    
    It remains to show that 
    \[
    	\D^{\C_1} \circ \Phi^* = \Phi^*\circ \D^{\C_2}.
    \]
    It is enough to check this equation on $f\in C^\infty_{\C_2^{\pi_2}}$ and $\pi_2^*\alpha$ for $\alpha\in \Omega^1_{TX_2}$:
    \begin{align*}
    	\langle \D^{\C_1} \circ \Phi^*f, \overline{X}\rangle &= \langle \d \Phi^*f, I_1\overline{X}\rangle = \langle \d f, \Phi_* I_1(\overline{X})\rangle \\
    	&= \langle \d f, I_2 \Phi_*(\overline{X})\rangle = \langle \D^{\C_2}f, \Phi_* (\overline{X})\rangle  = \langle \Phi^*\circ \d f, \overline{X} \rangle,
    \end{align*}
    for $\overline{X}\in \pi_1^*TX_1$. Also,
    \begin{align*}
    	\D^{\C_1} \circ \Phi^*\pi_2^*\alpha = \D^{\C_1}\pi_1^*\varphi^*\alpha = \pi_1^*\d \varphi^*\alpha = \Phi^*\pi_2^*\d \alpha = \Phi^* \D^{\C_2} \pi_2^*\alpha
    \end{align*}
    for $\alpha\in\Omega^1_{TX_2}$. This concludes the proof.
\end{proof}

The next lemma is a strengthening of parts of \cite[Proposition 4.6]{CattafiCrainicSalazar2020}.

\begin{lemma}\label{lem:MorphismsOfPfaffianFibrations}
Let $\pi_1\: P_1\to X_1$ and $\pi_2\: P_2\to X_2$ be submersions. 
\begin{enumerate}
    \item Suppose $(\Phi, \varphi)\: P_1\to P_2$ is a bundle map covering a local diffeomorphism $\varphi\: X_1\to X_2$. Then the induced map $(\Phi^{(1)}, \varphi)\: \:(J^1\pi_1, \C_1^1, \pi_1^{(1)})\to (J^1\pi_2, \C_2^1, \pi_2^{(1)})$ is a morphism of Pfaffian fibrations. 
    \item If $(\Phi, \varphi)\: (P_1, \C_1, \pi_1)\to (P_2, \C_2, \pi_2)$ is a morphism of Pfaffian fibrations covering a local diffeomorphism $\varphi\: X_1\to X_2$, then $(\Phi^{(1)}, \varphi)$ in (1) restricts to the partial prolongations 
    \[
        (\Phi^{(1)}, \varphi)\: (J^1_{\C_1}P_1, \C^1_1 , \pi^{(1)}_1)\to (J^1_{\C_2} P_2, \C^1_2, \pi_2^{(1)})
    \]
    which is a morphism of Pfaffian fibrations.
    \item The map $(\Phi^{(1)}, \varphi)$ restricts to
    \[
        (\Phi^{(1)}, \varphi)\: P^{(1)}_1\to P^{(1)}_2
    \]  
    which, when $(P_1, \C_1, \pi_1)$ and $(P_2, \C_2, \pi_2)$ are 1-integrable, is a morphism of Pfaffian fibrations.
\end{enumerate}
\end{lemma}

\begin{proof}
	Let us identify	$J^1\pi_i = \{ h_x\: T_{\pi_i(x)} X_i \to T_x P_i, \ | \ \mbox{splitting of $T_x\pi_i$} \}$. Then the map $\Phi^{(1)}\: J^1\pi_1\to J^1\pi_2$ is given by $\Phi^{(1)}(h_x) = T_x\Phi \circ h_x \circ \left(T_{\pi_i(x)} \varphi\right)^{-1}$. Let $p_1\: J^1\pi_i\to P_i$ be the projection and recall from Example \ref{ex:CartanForm} that
	\[
		X\in \C^1_i \mbox{ above $h_x\in J^1\pi_i$} \iff Tp_1(X) = h_x \circ T\pi^{(1)}_i(X).
	\]
	Therefore, when $X\in \C^1_1$, then
	\begin{align*}
		Tp_1(T\Phi^{(1)}(X)) &= T\Phi\circ  Tp_1(X)= T\Phi\circ h_x \circ T\pi_1^{(1)} (X)= \Phi^{(1)}(\sigma_x)\circ T\pi_2^{(1)} (T\Phi^{(1)}(X)), 
	\end{align*}	
	so indeed $T\Phi(X)\in \C^1_2$.
	
	For (2), it is clear that if $h_x$ is tangent to $\C_1$, then $\Phi^{(1)}(h_x)$ is tangent to $\C_2$, so there is a map $\Phi^{(1)}\: J^1_{\C_1} P_1\to J^1_{\C_2}P_2$, which is a morphism of Pfaffian fibrations by (1).
	
	Item (3) is more tricky because a direct proof using the curvatures typically involves $\Phi$-related sections of $\C_1$ and $\C_2$, which might not exist. Instead, we can use the interaction between Pfaffian fibrations and relative algebroids to apply what we know already for relative algebroids. 
	
	First, Proposition \ref{prop:NaturalityOfPfaffianRelativeAlgebroid} implies that $(\Phi, \varphi)$ induces a morphism of the relative algebroids from $(\pi_1^*TX_1, \onabla, \D^{\C_1})$ to $(\pi_2^*TX_2, \onabla, \D^{\C_2})$ that is \emph{fiberwise an isomorphism}. By Proposition \ref{prop:NaturalityOfProlongationsOfRelativealgebroids}, the morphism differentiates to a map between the prolongations $\Phi^{(1)}\: P_1^{(1)}\to P_2^{(1)}$. By Theorem \ref{thm:PfaffianFibrationsAsRelativeAlgebroids}, the prolongations of the relative algebroids agree with the prolongations of the Pfaffian fibrations, so indeed $\Phi^{(1)}$ restricts to a map
	\[
		(\Phi^{(1)}, \varphi)\: P^{(1)}_1\to P^{(1)}_2.
	\] 
	In case that both $P_1^{(1)}$ and $P_2^{(1)}$ are smooth, this is a morphism of Pfaffian fibrations by (2). 
\end{proof}

\section{Pfaffian groupoids}\label{sec:PfaffianGroupoids}

Next, we recall the basic notions of Pfaffian groupoids, introduced in \cite{Salazar2013}. In Section \ref{sec:PartialActionsAndPfaffianGroupoids}, we explain how (partial) action groupoids of a Pfaffian groupoid again a Pfaffian groupoid.

\subsection{Pfaffian groupoids}
A Pfaffian groupoid is a Lie groupoid that is also a Pfaffian fibration which is multiplicative in the following sense.
\begin{definition}[\cite{Salazar2013}, Definition 6.1.4]\label{def:PfaffianGroupoid}
    A \textbf{Pfaffian groupoid} $(\GG, \HH)\toto M$ consists of a Lie groupoid $\GG\toto M$ with a multiplicative distribution $\HH\subset T\GG$ for which the source map $s\: (\GG, \HH)\to M$ is a Pfaffian fibration.  
\end{definition}

A local bisection $\sigma\in \Bisloc(\GG)$ of a Pfaffian groupoid $(\GG, \HH) \toto M$ is \textbf{holonomic} when $\im T\sigma \subset \HH$. The collection of local holonomic bisections $\Bisloc(\GG, \HH)$ form what's called a \textbf{generalized pseudogroup}, meaning that it is closed under restriction, inversion and composition, and satisfies a gluing axiom.

\begin{example}[Jet groupoids]\label{ex:JetGroupoids}
    The first jet $J^1\GG\toto M$ of a groupoid $\GG\toto M$ consists of first jets of local bisections of $\GG$. The composition is induced by composition of bisections:
    $j^1_{y}\tau\cdot j^1_x\sigma= j^1_x(\tau\cdot \sigma)$. Alternatively, $J^1\GG$ consists of pointwise splittings $h\: T_{s(g)}M\to T_g\GG$ of $T_gs$ for which $T_g t \circ h$ is an isomorphism. There is a natural groupoid projection $p_1\: J^1\GG\to \GG$ that sends $j^1_x \sigma$ to $\sigma(x)$. Note that $J^1\GG$ is an open subset of $J^1s$.

    The first jet groupoid $J^1\GG\toto M$ is naturally a Pfaffian groupoid $(J^1\GG, \HH^1)\toto M$ with the Pfaffian distribution $\HH^1$ given by the Cartan distribution on $J^1s$ (Example \ref{ex:CartanForm}) restricted to $J^1\GG$. Concretely, 
    \begin{equation}\label{eq:CharacterizationCartanDistributionJetGroupoid}
    	X\in \HH^1 \mbox{ above $j^1_x\sigma$} \iff Tp_1(X) = T_x\sigma \circ Ts(X).
    \end{equation}
    
    The higher jet groupoids of a Lie groupoid $\GG\toto M$ are also naturally Pfaffian groupoids
    \[
    	(J^k\GG, \HH^k)\toto M
    \]
    by restricting the Cartan distribution on $J^ks$ to $J^k \GG\subset J^ks$. 
\end{example}

\begin{example}[Jets of local diffeomorphisms]\label{ex:CartanDistributionDiffeomorphisms} Let $\Diffloc(M)$ be the pseudogroup of local diffeomorphisms on $M$. Any local diffeomorphisms $\varphi$ can be regarded as local bisections of the pair groupoid $M\times M\toto M$, through $(\varphi, \id)\: M\to M\times M$. This way, we get an identification
\[
	J^1\Diffloc(M) \cong J^1(M\times M)
\]
and so we can regard $(J^1\Diffloc(M), \HH^1)\toto M$ as a Pfaffian groupoid. In this case, purely in terms of diffeomorphisms, the Cartan distribution can be characterized as
\begin{equation}\label{eq:CharacterizationCartanDistributionDiffeomorphisms}
	X\in \HH^1 \mbox{ above $j^1_x\varphi$} \iff Tt(X) = T_x\varphi \circ Ts(X),
\end{equation}
with $j^1_x\varphi\in J^1\Diffloc(M)$. 

Often, the diffeomorphisms preserving a geometric structure are given by a PDE $\GG\subseteq J^k\Diffloc(M)$. By restricting the Cartan distribution $\HH^k$ to $\GG$, we obtain a Pfaffian groupoid $(\GG, \HH^k)\toto M$ whose holonomic bisections correspond to the symmetries of the geometric structure.
\end{example}

\subsection{Prolongations of Pfaffian groupoids}\label{subsec:ProlongationOfPfaffianGroupoids}
As for Pfaffian fibrations, there is a one-form (with coefficients, now multiplicative) on $\GG$ whose kernel is exactly $\HH$, given by the projection
\[
	\omega_\HH\: T\GG\to T\GG/\HH.
\] 
The \textbf{partial prolongation} of a Pfaffian groupoid $(\GG, \HH)\toto M$ is
\[
	J^1_\HH\GG = \{ h\in J^1\GG \ | \ h^*\omega_\HH = 0\}.
\]
Equipped with the Cartan distribution $\HH^1$ inherited from $J^1\GG$, we obtain a new Pfaffian groupoid $(J^1_\HH\GG, \HH^1)\toto M$. 

The \textbf{prolongation} of $(\GG, \HH)\toto M$ consists of those elements with vanishing curvature:
\[
	\GG^{(1)} = \{ h \in J^1\GG\ | \ h^*\omega_\HH =0, \ h^*\kappa_\HH =0\}
\]
If smooth, it becomes a Pfaffian groupoid $(\GG^{(1)}, \HH^1)\toto M$ when equipped with the Cartan distribution $\HH^1$ restricted from $J^1\GG$.

Like for Pfaffian fibrations, we can talk about $k$-integrability and formal integrability for Pfaffian groupoids.

The \textbf{reduced Pfaffian groupoid} of $(\GG, \HH)$ is the locus where the intrinsic torsion vanishes (cf. Remark \ref{rk:reducedFibrationIntrinsicTorsion}). It can also be defined as the image of the prolongation space under the projection
\[
	\overline{\GG} = p_1(\GG^{(1)}).
\]
If smooth, we obtain a new Pfaffian groupoid $(\overline{\GG}, \HH)\toto M$. If $\GG^{(1)}$ is in addition smooth, then $(\overline{\GG}, \HH)\toto M$ is 1-integrable, even though $(\GG, \HH)$ may not be.

\begin{remark}
    We will occasionally work with the first prolongation space $\GG^{(1)}$ of a Pfaffian groupoid even when it is not smooth. Notably, we will use the first prolongation of the groupoid of Pfaffian symmetries, introduced later, that is not assumed to be smooth.
\end{remark}

\subsection{Morphisms of Pfaffian groupoids}

While it is possible to consider more general morphisms of Pfaffian groupoids over different bases, the following notion is sufficient for the purposes of this paper.

\begin{definition}
    A \textbf{morphism of Pfaffian groupoids} (covering the identity) from $(\GG_1, \HH_1)\toto M$ to $(\GG_2, \HH_2)\toto M$ is a groupoid morphism $\Phi\:\GG_1\to \GG_2$ covering the identity such that $T\Phi(\HH_1)\subseteq \HH_2$. 
\end{definition}

\begin{lemma}\label{lem:MorphismsOfPfaffianGroupoids}
    The following are true.
    \begin{enumerate}
        \item If $\Phi\: \GG_1\to \GG_2$ is a morphism of groupoids (covering the identity), then $\Phi_*\: (J^1\GG_1, \HH^1_1) \to (J^1\GG_2, \HH^1_2)$ is a morphism of Pfaffian groupoids.
        \item If $\Phi\: (\GG_1, \HH_1)\to (\GG_2, \HH_2)$ is a morphism of Pfaffian groupoids, then $\Phi^{(1)}\: J^1\GG_1\to J^1\GG_2$ restricts to the partial prolongations
        \[
            \Phi^{(1)}\: (J^1_{\HH_1}\GG_1, \HH^1_1)\to (J^1_{\HH_2} \GG_2, \HH^1_2)
        \]
        as a morphism of Pfaffian groupoids.
        \item The map $\Phi^{(1)}$ also restricts to the prolongation spaces
        \[
            \Phi^{(1)}\: \GG_1^{(1)}\to \GG_2^{(1)}
        \]
        which, when both $(\GG_1, \HH_1)$ and $(\GG_2, \HH_2)$ are 1-integrable, is a morphism of Pfaffian groupoids.
    \end{enumerate}
\end{lemma}
\begin{proof}
	This is the content of Lemma \ref{lem:MorphismsOfPfaffianFibrations}, restricted to the (partial) prolongation of the Pfaffian groupoids.
\end{proof}

\subsection{Partial actions and Pfaffian groupoids}\label{sec:PartialActionsAndPfaffianGroupoids}

In Section \ref{sec:ActionsByPfaffianGroupoidsOnPfaffianFibrations}, we will encounter natural groupoid actions that are \textit{partial}, as opposed to being complete. Partial actions are different from local actions--partial refers to the space the groupoid acts on, not to the groupoid itself. Every partial action still has a well-defined action groupoid. We spend a couple of words on partial actions of groupoids, as discussed in \cite{AnantharamanDelaroche2020, MarinPinedoRodrigue2025}.
\begin{definition}\label{def:PartialAction}
	A \textbf{partial action} of a Lie groupoid $\GG\toto M$ on a map $\mu\: P\to M$ (denoted as $m\:\GG \paction P$) consists of a local map $m\: \GG\leftindex_{s}{\times}_\mu P \supseteq \dom m \to P$, $(g, p) \mapsto g\cdot p$, satisfying the following properties:
	\begin{enumerate}
		\item For all $x\in P$, $(u(\mu(x)), x)\in \dom m$ and $u(\mu(x))\cdot x = x$, where $u\: M\to \GG$ is the unit map.
		\item If $(g, x)\in \dom m$, then $(g^{-1}, g\cdot x)\in \dom m$.
		\item If $(g_0, x), (g_1, g_0\cdot x)\in \dom m$, then $(g_1g_0, x)\in \dom m$ and $(g_1g_0)\cdot x = g_1 \cdot (g_0\cdot x)$.
	\end{enumerate}
\end{definition}
The domain of the multiplication $m$ is assumed to be open (hidden in the notion of a local map). The conditions imply that $\GG\pltimes P := \dom m\toto P$ is a Lie groupoid, which we call the \textbf{(partial) action groupoid}.

Perhaps surprisingly, the partial action groupoid $\GG\pltimes P \toto P$ is a \emph{global} Lie groupoid, and not just a local one. A partial action is distinct from a \emph{local action} of a Lie groupoid, where the locality is concentrated around the unit section of the groupoid (which would give rise to a local action groupoid).

If $\Phi\: \GG_1 \to \GG_2$ is a Lie group morphism and  $\GG_2\paction P$ is a partial action, then it induces naturally a partial action $\GG_1 \paction P$.

In Section \ref{sec:ActionsByPfaffianGroupoidsOnPfaffianFibrations} we will encounter partial actions of Pfaffian groupoids. If $(\GG, \HH)\toto M$ is a Pfaffian groupoid and $\GG\paction P$ a partial action along $\mu\: P\to M$, then we can pull back the distributions $\HH\subset T\GG$ to the action groupoid $\GG\pltimes P$ along the projection $\pr_\GG\: \GG\pltimes P \to \GG$. Concretely,
\[
\pr_\GG^*\HH = \left\{ (X, Y)\in T\GG\leftindex_{Ts}{\times}_{T\mu} TP\big\vert_{\GG\pltimes P} \ \big\vert \ X\in \HH \right\}\subset T(\GG\pltimes P).
\]
If $\omega_\HH$ is the form corresponding to $\HH$, note that $\pr_\GG^*\HH = \ker \pr_\GG^*\omega_\HH$. After (canonically) identifying $T(\GG\pltimes \HH)/ \pr_\GG^*\HH\cong \pr^*_\GG\left(T\GG/\HH\right)$, it actually holds that
\[
	\pr^*_\GG\omega_\HH = \omega_{\pr^*_\GG \HH}.
\]

\begin{lemma}\label{lem:ActionPfaffianGroupoid}
	Let $\GG\toto M$ be a groupoid with a partial action on $P$ along $\mu\: P\to M$.
	\begin{enumerate}
		\item If $(\GG, \HH)\toto M$ is a Pfaffian groupoid, then $(\GG\pltimes P, \pr_\GG^*\HH)\toto P$ is a Pfaffian groupoid.
		
		\item The canonical immersion $((J^1\GG)\pltimes P, \pr_{J^1\GG}^*\HH^1)\to (J^1(\GG\pltimes P), \HH^1)$ is a morphism of Pfaffian groupoids.
		\item If $(\GG, \HH)\toto M$ is a Pfaffian groupoid, then there is a morphisms of Pfaffian groupoids
		\[
		((J^1_{\HH}\GG)\pltimes P, \pr^*_{J^1_\HH\GG}\HH^1)\to (J^1_{\pr^*_\GG\HH}(\GG\pltimes P), \HH^1),
		\]
		\item There is a canonical morphism
		\[
		\GG^{(1)}\pltimes P\to (\GG\pltimes P)^{(1)}
		\]
		which is a morphism of Pfaffian groupoids when both $(\GG, \HH)$ and $(\GG\pltimes P, \pr_\GG^*\HH)$ are 1-integrable.
	\end{enumerate}
\end{lemma}

\begin{proof}
	Let us start with (1). The distribution $\pr^*_\GG \HH\subset T(\GG\pltimes P)$ is transverse to the source fibers of $\GG\pltimes P$ because $\HH$ is. The distribution $(\pr^*_\GG \HH)^s $ is given by $\HH^s\leftindex_{Ts}{\times}_{T\mu} 0\vert_{\GG\pltimes P}$, which is involutive, so $(\GG\pltimes P, \pr^*_\GG \HH)\toto P$ is a Pfaffian groupoid.
	
	For (2), the morphism $i\: (J^1\GG)\pltimes P\to J^1(\GG\pltimes P)$ is the first jet of the (partially defined) map $\mu^*\: \Bisloc(\GG)\dasharrow \Bisloc(\GG\pltimes P)$, with $\mu^*\sigma(x) = (\sigma(\mu(x)), x)$. Concretely, on the level of first jets, it is given by $i(j^1_{\mu(x)}\sigma, x) = j^1_x(\mu^*\sigma)$. This is well-defined, because $(j^1_{\mu(x)}\sigma, x)\in (J^1\GG)\pltimes P$ if and only if $(\sigma(\mu(x)), x)\in \GG\pltimes P$. The following diagram commutes:
	\[
	\begin{tikzcd}
		(J^1\GG)\pltimes P \arrow[rr, "i"] \arrow[dr, "{(p_1, \id)}"'] & & J^1(\GG\pltimes P) \arrow[dl, "p_1"] \\
		& \GG\pltimes P,  &
	\end{tikzcd}
	\]
	so, if $(X, Y)\in \pr_\GG^*\HH$ above $(j^1_{\mu(x)}\sigma, x)$, then
	\begin{align*}
		Tp_1 Ti(X, Y) &= (Tp_1(X), Y) = (T\sigma Ts(X), Y) = (T\sigma T\mu(Y), Y) = T(\mu^*\sigma) Ts(X, Y),
	\end{align*}
	which shows by Equation (\ref{eq:CharacterizationCartanDistributionJetGroupoid}) that $Ti(X, Y)\in \HH^1$. We conclude that the inclusion is a morphism of Pfaffian groupoids.
	
	For (3), note that $i$ restricts to $i\: (J^1_\HH\GG)\pltimes P \to J^1_{\pr^*_\GG \HH}(\GG\pltimes P)$ because if $\sigma\in \Bisloc(\GG)$, then $(\mu^*\sigma)^*\pr_\GG^*\omega_\HH = \mu^*\sigma^*\omega_\HH$, so the first jet of $\mu^*\sigma$ is tangent to $\ker \pr_\GG^*\omega_\HH$ whenever $j^1_{\mu(x)}\sigma$ is tangent to $\HH$.
	
	For item (4), the map will restrict to a map between the prolongation spaces if we can show that the curvatures are related through
	\begin{equation}\label{eq:curvaturesOfActionGroupoids}
		\kappa_{\pr^*_\GG \HH} = \pr^*_\GG \kappa_\HH,
	\end{equation}
	where we canonically identify $T(\GG\pltimes \HH)/ \pr_\GG^*\HH\cong \pr^*_\HH\left(T\GG/\HH\right)$. This equation follows from the fact that $T\pr_\GG\: \pr^*_\GG \HH\to \HH$ is surjective. This implies that, given $(X_1, Y_1)\in \pr_\GG^*\HH$, we can find extensions $(\tilde{X}_i, \tilde{Y}_i)\in \Gamma(\pr_\GG^*\HH)$ that are $\pr_\GG$-related to $\tilde{X}_i\in \Gamma(\HH)$. Therefore, the curvatures satisfy:
	\[
	\kappa_{\pr^*_\GG\omega}((X_1, Y_1), (X_2, Y_2)) = \pr^*_\GG\omega_\HH([(\tilde{X}_1, \tilde{Y}_1), (\tilde{X}_2, \tilde{Y}_2)]) = \omega_\HH([\tilde{X}_1, \tilde{X}_2]) = \kappa_\HH(X_1, X_2),
	\]
	so Equation (\ref{eq:curvaturesOfActionGroupoids}) follows. If smooth, the map is a morphism of Pfaffian groupoids by (2).
\end{proof}

\section{Symmetries of Pfaffian fibrations}\label{sec:SymmetriesOfPfaffianFibrations}

We have now arrived at the second part of the paper, where we investigate the symmetries of a Pfaffian fibration, and explain how those symmetries give rise to symmetries of the underlying relative algebroid.

\subsection{Internal symmetries of Pfaffian fibrations} The most general notion of symmetry of a Pfaffian fibration is one that preserves the Cartan distribution. More precisely, the \textbf{pseudogroup of internal symmetries} of a Pfaffian fibration $(P, \C, \pi)$ is given by
\[
	\Diffloc(P, \C) := \left\{ \varphi\in \Diffloc(P) \ | \ T\varphi (\C) \subseteq \C\right\}.
\]
Here, $\Diffloc(P)$ is the pseudogroup of local diffeomorphisms on $P$. If $(P, \C, \pi)$ is comes from a PDE (Example \ref{ex:CartanForm}), this notion corresponds to the notion of \emph{internal symmetries} for PDEs \cite{AndersonKamranOlver1993}.

An internal symmetry $\varphi$ of $(P, \C, \pi)$ sends solutions to solutions, but only when it well-defined. This is already visible at the level of first jets. The \textbf{prolongation} of an internal symmetry is the local diffeomorphism
\[
	\varphi^{(1)}\: J^1_\C P\dasharrow J^1_\C P, \quad \varphi^{(1)}\: P^{(1)}\to P^{(1)}
\]
given by $\varphi^{(1)} (H_x) = T_x\varphi(H_x)$ for $H_x\subset \C_x$ with $H_x\oplus \C^\pi_x = \C_x$. This is only well-defined when $T_x\varphi(H_x) + \C^\pi_{\varphi(x)} = \C_{\varphi(x)}$. If instead we identity elements of $J^1_\C P$ with splittings $h_x\: T_{\pi(x)}\to \C_x$, then 
\[
	\varphi^{(1)}(h_x)= T_x\varphi \circ h_x \circ \left( T_{\varphi(x)}\pi\circ T_x\varphi \circ h_x \right)^{-1}.
\]  
The domain of $\varphi^{(1)}$, given by the subset where $T\pi\circ T\varphi \circ h$ is invertible, is open and dense over the domain of $\varphi$ (if $(P, \C, \pi)$ is involutive, this is also true for $\varphi^{(1)}$ restricted to $P^{(1)}$). 

\subsection{Pfaffian symmetries of Pfaffian fibrations}\label{subsec:SymmetriesOfPfaffianFibrations}

A \textbf{Pfaffian symmetry} of a Pfaffian fibration $(P, \C, \pi)$ is an internal symmetry that also preserves $\C^\pi$. They form the \textbf{pseudogroup of Pfaffian symmetries}, given by
\[
    \Diffloc(P, \C, \pi) := \left\{ \varphi\in \Diffloc(P) \ \big\vert \  T\varphi(\C) \subseteq \C, \ T\varphi(\C^\pi)\subseteq \C^\pi \right\}.
\]

\begin{remark}
	Note that elements of $\Diffloc(P, \C, \pi)$ need not necessarily preserve $\ker T\pi$ and therefore may not cover a local diffeomorphism on the base $X$. Therefore, a local symmetry of a Pfaffian fibration is \textit{not necessarily} a morphism of Pfaffian fibrations in the sense of Definition \ref{def:morphismOfPfaffianFibration}.
\end{remark}

Since a local symmetry $\varphi\in \Diffloc(P, \C, \pi)$ is in particular an internal symmetry, it prolongs to a local diffeomorphism
\[
	\varphi^{(1)}\: J^1_\C P \dasharrow J^1_\C P, \quad \varphi^{(1)}\: P^{(1)}\dasharrow P^{(1)},
\]
but this time, since the $\varphi$ also preserves $\C^\pi$, the transversality condition $T_x\varphi(H_x) + \C^\pi_{\varphi(x)}= \C_{\varphi(x)}$ is automatically satisfied, so the domain of $\varphi^{(1)}$ is exactly $J^1_\C P$ restricted to the domain of $\varphi$.

The following result relates internal symmetries to Pfaffian symmetries of the prolongation.
\begin{theorem}\label{thm:InternalAndPfaffianSymmetries}
	Let $(P, \C, \pi)$ be a Pfaffian fibration.
	\begin{enumerate}
		\item If $\varphi\in \Diffloc(P, \C)$ is an internal symmetry, then $\varphi^{(1)}\: J^1_\C P \dasharrow J^1_\C P$ is a Pfaffian symmetry of $(J^1_\C P, \C^1, \pi^{(1)})$.  
		\item If $(P, \C, \pi)$ is 1-integrable, then $\varphi^{(1)}\: P^{(1)}\dasharrow P^{(1)}$ is also a Pfaffian symmetry. 
		\item If $\tilde{\varphi}$ is a Pfaffian symmetry of $J^1_\C P$, then it is locally the prolongation of an internal symmetry of $(P, \C, \pi)$.
		\item If $(P, \C, \pi)$ is 1-integrable with \emph{involutive} tableau map, then every Pfaffian symmetry $\tilde{\varphi}$ of $(P^{(1)}, \C^1, \pi^{(1)})$ is locally the prolongation of an internal symmetry of $(P, \C, \pi)$. 
	\end{enumerate}
\end{theorem}
\begin{proof}
	The prolongation $\varphi^{(1)}\: J^1_\C P \dasharrow J^1_\C P$ covers the internal symmetry $\varphi$, i.e. $p_1\circ \varphi^{(1)} = \varphi \circ p_1$, so it preserves $\ker Tp_1 = (\C^1)^{\pi^{(1)}}$. It remains to show that $\varphi^{(1)}$ also preserves $\C^1$. The Cartan distribution above $H_x\in J^1P$ is given by
	\[
		\C^1_{H_x}= (Tp_1)^{-1}(H_x). 
	\]
	It follows, for $H_x\subseteq \C_x$ in the domain of $\varphi^{(1)}$, that 
	\[
		T\varphi^{(1)}(\C^1_{H_x})=  T \varphi^{(1)}\left( (Tp_1)^{-1}(H_x)\right)= (Tp_1)^{-1}\left( T\varphi(H_x)\right)=  \C^1_{\varphi^{(1)}(H_x)},
	\]
	so $\varphi^{(1)}$ is indeed a Pfaffian symmetry. 
	
	Statement (2) follows immediately from restriction.
	
	For (3), if $\tilde{\varphi}$ is a Pfaffian symmetry of $(J^1_\C P, \C^1, \pi^{(1)})$, then it preserves in particular $\ker Tp_1 = (\C^1)^{\pi^{(1)}}$ so locally it is covered by a diffeomorphism $\varphi$ on $P$: $\varphi\circ p_1 = p_1\circ \tilde{\varphi}$. First, we show that $\varphi$ is an internal symmetry of $(P, \C, \pi)$. We need the following claim:
	\begin{claim*}
		Let $U_x\subset (J^1_\C P)_x$ be any open subset. Then
		\[
			\C_x = \operatorname{span}(H_x \ | \ H_x\in U_x).
		\]
	\end{claim*}
	\begin{proof}[Proof of claim] Recall that $(J^1_\C P)_x \cong \{ h_x\in \Hom(TX, \C_x) \ | \  T\pi\circ h_x = \id_{TX} \}$. Fix $\tilde{h}_x\in U_x$ and set $\tilde{H}_x = \im \tilde{h}_x$. Then $\C_x = \tilde{H}_x \oplus \C^\pi_x$. It remains to show that $\C^\pi_x$ is in the span. For each $Y\in \C^\pi_x$ there exists $\xi_x\in \Hom(TX, \C^\pi_x)$ such that $\tilde{h}_x + \xi_x\in U_x$ and $\xi_x(X) = Y$ for some $X\in TX$. It follows that $Y = (\tilde{h}_x + \xi_x)(X) - \tilde{h}_x(X)$ is in $\operatorname{span}(\im h_x \ | \ h_x\in U_x)$, which proves the claim. \phantom{\qedhere}
	\end{proof}
	The proof of statement (3) proceeds as follows. If $H_x$ is in the domain of $\tilde{\varphi}$, then
	\[
		T\varphi(H_x) = Tp_1 \circ T\tilde{\varphi}(\C^1_{H_x}) = Tp_1(\C_{\tilde{\varphi}(H_x)}) = \tilde{\varphi}(H_x).
	\]
	It follows that $T\varphi$ preserves $\C_x$ (from the claim) and that $\tilde{\varphi}$ coincides with the prolongation of $\varphi$ on its domain.
	
	Statement (4) can be proven in the same way under the condition that a similar claim holds for $P^{(1)}$.
	\begin{claim*}
		Suppose that the tableau map $\tau_x$ is involutive at $x\in P$ (see Appendix \ref{app:TableauMapsAndInvolutivity}) and $U_x$ is an open neighborhood of $(P^{(1)})_x$. Then 
		\[
			\C_{x} = \operatorname{span}(H_x \ | \ H_x\in U_x). 
		\]
	\end{claim*}
	\begin{proof}[Proof of claim]
		This time $(P^{(1)})_x$ is an affine subspace of $\Hom(TX, \C_x)$ modeled on $(\C^\pi_x)^{(1)}$. Fix $\tilde{h}_x\in U_x$. By Proposition \ref{prop:CartansBoundTableauMaps},  there is a basis $(e_k)_k$ for which $\iota_{e_1}\: (\C^{\pi})^{(1)}\to \C^\pi$ is surjective. It follows that for each $Y\in \C^\pi_x$ there is an element $\xi_x\in (\C^\pi_x)^{(1)}$ such that $\xi_x(X) = Y$ for some $X\in TX$ and $\tilde{h}_x + \xi_x\in U_x$. We conclude that $\C_x^\pi$ is contained in the span. Therefore,
		$\C_x = \im \tilde{h}_x \oplus \C^\pi_x \subseteq \operatorname{span}(\im h_x \ | \ h_x\in U_x)$, proving the claim. \phantom{\qedhere}
	\end{proof}	
	We can finish the proof by proceeding as for statement (3). 
\end{proof}

\begin{remark}
	According to the Lie-B\"{a}cklund theorem (see e.g. \cite{AndersonIbragimov1979}), the internal symmetries of $J^kq$ for a submersion $q\: Q\to X$ are automatically Pfaffian and in fact correspond to prolongations of point symmetries (diffeomorphisms on $Q$), whenever the dimension of the fibers of $q$ is greater than one.
\end{remark}

\subsection{Pseudogroup symmetries of the underlying relative algebroid}
The symmetries of a Pfaffian fibration are defined in such a way that it interacts well with the underlying relative algebroid. 

\begin{proposition}\label{prop:PfaffianSymmetriesInduceAlgebroidSymmetries}
    The Pfaffian symmetries of a Pfaffian fibration $(P, \C, \pi)$ induce (local) isomorphisms of the underlying relative algebroid $(\pi^*TX, \onabla, \D^\C)$. 
\end{proposition}

\begin{proof}
    Let $\varphi\in \Diffloc(P, \C, \pi)$ be a symmetry. By definition, the tangent map $T\varphi$ preserves both $\C$ and $\C^\pi$, and therefore descends to a bundle map $\overline{T\varphi}\: \C/\C^\pi \to  \C/\C^\pi$. The induced map on $\nu(\C^\pi)$ is also denoted by $\overline{T\varphi}$. Using the identification $\C/\C^\pi \cong \pi^*TX$, we obtain a map $\Phi\: \pi^*TX\to \pi^*TX$ covering $\varphi$.   
    
    By definition, $\varphi$ is a map of foliations. We will first show that $(\Phi, \varphi)$ is a map of flat foliated vector bundles (Equation (\ref{eq:MapOfFoliatedVectorBundles})). Recall the definition of $\onabla$ from Equation (\ref{eq:FlatConnectionPfaffian}). The map $\overline{T\varphi}$ intertwines both the inclusion $I\:\pi^*TX\hookrightarrow \nu(\C^\pi)$ and the projection $\Pi\: \nu(\C^\pi)\to \C/\C^\pi$. Taking $X\in \Gamma(\C^\pi)$, $Y\in \Gamma(\C)$, it follows that
    \begin{align}\label{eq:ConnectionPreserved}
    \begin{split}
        \Phi^*\left(\onabla_X \overline{Y}\right) &= \Phi^* \Pi\left(\onabla^{\text{Bott}}_X \overline{Y}\right)
        = \Pi \left( \overline{\varphi^*} \left( \onabla^{\text{Bott}}_X \overline{Y} \right)\right) \\
        &= \Pi \left( \overline{\varphi^*([X, Y])}\right) = \Pi \left( \overline{[\varphi^*(X), \varphi^*(Y)]} \right) \\
        &= \Pi \left( \onabla^{\text{Bott}}_{\varphi^*(X)} \overline{\varphi}^* \overline{Y}\right) = \onabla_{\varphi^*(X)} \Phi^*(\overline{Y}),
    \end{split}
    \end{align}
    so $\Phi$ is indeed a map of foliated vector bundles. 

    One shows that $\Phi$ preserves brackets in a similar manner: if $X, Y\in \Gamma(\C)$, then
    \begin{align}\label{eq:BracketPreserved}
        \Phi^*[\overline{X}, \overline{Y}]^\C &=\Phi^* \left(\Pi(\overline{[X, Y]})\right) 
        = \Pi \left( \overline{[\varphi^*X, \varphi^*Y]}\right) = [\Phi^*\overline{X}, \Phi^*\overline{Y}]^\C.
    \end{align}
    If follows that $(\Phi, \varphi)$ is an isomorphism of relative algebroids.
\end{proof}
\begin{remark}
	Internal symmetries do not directly act on the relative algebroid. Their prolongations however, which are Pfaffian symmetries of the prolongation, act by symmetries of the algebroid underlying the prolongation.
\end{remark}

\subsection{Groupoids of symmetries of Pfaffian fibrations}\label{sec:GroupoidSymmetriesOfPfaffianFibrations}

The pseudogroups of internal and Pfaffian symmetries are determined by their first jet. So, for a Pfaffian fibration $(P, \C, \pi)$  it is perhaps more interesting to instead consider the Pfaffian groupoids
\begin{align*}
	\GG_\C &:= \left\{ j^1_x\varphi\in J^1\Diffloc(P)\ \big\vert \  T_x\varphi(\C_x) \subseteq \C_{\varphi(x)} \right\}, \\ 
    \GG_{\C, \pi} &:= \{ j^1_x\varphi\in \GG_\C\  \big\vert  \  T_x\varphi(\C^\pi_x) \subseteq \C^\pi_{\varphi(x)} \},
\end{align*}
equipped with the Cartan distribution $\HH^1$ on $J^1\Diffloc(P)$ (Example \ref{ex:CartanDistributionDiffeomorphisms}). Their holonomic bisections have the property that
\[
    \Bisloc(\GG_\C, \HH^1)\cong \Diffloc(P, \C), \quad \Bisloc(\GG_{\C, \pi}, \HH^1) \cong \Diffloc(P, \C, \pi).
\]
Note that, in general, $J^1\Diffloc(P, \C)$ (resp. $J^1\Diffloc(P, \C, \pi)$) is strictly contained in $\GG_\C$ (resp. $\GG_{\C, \pi}$).

\subsection{Groupoids of symmetries and the underlying relative algebroid}
It is the groupoid of Pfaffian symmetries $(\GG_{\C, \pi}, \HH^1)\toto P$ of a Pfaffian fibration $(P, \C, \pi)$ that interacts well with the underlying relative algebroid structure, as we will discuss next.

The canonical representation $\GG_{\C, \pi}\curvearrowright TP$ preserves both $\C$ and $\C^\pi$, and thus descends to a representation $\GG_{\C, \pi}\curvearrowright \C/\C^\pi\cong \pi^*TX$. 

\begin{theorem}\label{thm:PfaffianDerivationInvariant}
    The action $(\GG^{(1)}_{\C, \pi}, \GG_{\C, \pi})\curvearrowright (\pi^*TX, \onabla, \D^\C)$ is a representation by symmetries of relative algebroids (Section \ref{sec:RepresentationOnRelativeAlgebroids}). 
\end{theorem}

\begin{proof}
    First, we show that the action $(\GG^{(1)}_{\C, \pi}, \GG_{\C, \pi})\curvearrowright (\pi^*TX, \onabla)$ is a representation on a flat foliated vector bundle (Definition \ref{def:RepresentationOnFlatFoliatedVectorBundle}): Equation \ref{eq:ConnectionPreserved} holds pointwise for $j^2_x\varphi\in \GG^{(1)}_{\C, \pi} \subset J^2\Diffloc(P)$ and the first jets of $X$ and $Y$, so $\GG^{(1)}_{\C, \pi}$ preserves the PDE of flat sections $J^1_{\onabla} \pi^*TX$ of $(\pi^*TX, \onabla)$. Consequently, by Proposition \ref{prop:RepresentationsOfFoliatedVectorBundles} there is a canonical representation on $\DD^1_{(\pi^*TX, \onabla)}$.
    
	The section $\D^\C$ is invariant for this representation because Equation \ref{eq:BracketPreserved} holds pointwise for elements in $\GG^{(1)}_{\C, \pi}$.
\end{proof}

\begin{remark}
    Since Theorem \ref{thm:PfaffianDerivationInvariant} is pointwise in nature, it is true even if $\GG^{(1)}_{\C, \pi}$ is not smooth.
\end{remark}

\section{Actions by Pfaffian groupoids on Pfaffian fibrations}\label{sec:ActionsByPfaffianGroupoidsOnPfaffianFibrations}
In practice, rather than considering the full groupoid of symmetries of a Pfaffian fibration, one is often interested specific symmetries of the PDE, e.g. diffeomorphisms or gauge symmetries. Such symmetries are often described by a certain action of a Pfaffian groupoid on a Pfaffian fibration.

Analogous to the two types of symmetries of Pfaffian fibrations, there will be two types of actions by Pfaffian groupoids: actions \textbf{by internal symmetries} and actions \textbf{by Pfaffian symmetries}. The groupoids acting by internal symmetries, however, only prolong to \emph{partial groupoid actions} on the prolongation.

\subsection{Actions by internal symmetries}
Let $(\GG, \HH)\toto M$ be a Pfaffian groupoid and $(P, \C, \pi)$ a Pfaffian fibration. We say that a partial action $m\: \GG\paction P$ along a moment map $\mu\: P\to M$ is \textbf{by internal symmetries} when the action map satisfying
\begin{equation}\label{eq:DefinitionInternalAction}
	Tm\left( \HH \leftindex_{Ts}{\times}_{T\mu} \C \big\vert_{\dom m} \right)\subseteq \C.
\end{equation}
\begin{lemma}
	If $(\GG, \HH)\paction (P, \C, \pi)$ is a partial action by internal symmetries, then the local holonomic bisections of $(\GG, \HH)$ act by local internal symmetries on $(P, \C, \pi)$. 
\end{lemma}
\begin{proof}
	For a local holonomic bisection $\sigma\in \Bisloc(\GG, \HH)$, the local diffeomorphism $L_\sigma\: P \dasharrow P$ obtained from left translations, i.e. $L_\sigma(x) = m(\sigma(\mu(x)), x)$, is an internal symmetry by Equation \ref{eq:DefinitionInternalAction}.
\end{proof}

\begin{lemma}
	Let $\Phi\: (\GG_1, \HH_1)\to (\GG_2, \HH_2)$ be a morphism of Pfaffian groupoids (covering the identity). If $(\GG_2, \HH_2)\paction (P, \C, \pi)$ is a partial action by internal symmetries, then so is the induced action. 
\end{lemma}

\begin{example}
	The groupoid $\GG_\C \toto P$ of internal symmetries of a Pfaffian fibration $(P, \C, \pi)$ acts on $(P, \C, \pi)$ by internal symmetries. 
\end{example}

\subsection{Actions by Pfaffian symmetries}
Let $(\GG, \HH)\toto M$ be a Pfaffian groupoid and $(P, \C, \pi)$ a Pfaffian fibration. We say that $(\GG, \HH)$ \textbf{acts} on $P$ along $\mu\: P\to M$ \textbf{by Pfaffian symmetries} when the (partial) action map $m\: \GG \leftindex_{s}{\times}_\mu P\dasharrow P$ satisfies
\begin{equation}\label{eq:DefinitionGroupoidActionByPfaffianSymmetries}
    Tm\left(\HH \leftindex_{Ts}{\times}_{T\mu} \C\vert_{\dom m} \right)\subseteq \C, \quad  Tm\left( \HH \leftindex_{Ts}{\times}_{T\mu} \C^\pi \vert_{\dom m} \right) \subseteq \C^\pi. 
\end{equation}

\begin{proposition}
    If $(\GG, \HH)\toto M$ acts by Pfaffian symmetries on $(P, \C, \pi)$ along a map $\mu\: P\to M$. Then the pseudogroup of holonomic bisection $\Bisloc(\GG, \omega)$ acts on $P$ by Pfaffian symmetries of $(P, \C, \pi)$. 
\end{proposition}

\begin{proof}
    For $\sigma\in \Bisloc(\GG)$, let $L_\sigma\in \Diffloc(P)$ be the local diffeomorphism defined as $L_\sigma(x) = m(\sigma(\mu(x)), x)$. Equation (\ref{eq:DefinitionGroupoidActionByPfaffianSymmetries}) implies that $L_\sigma\in \Diffloc(P, \C, \pi)$ whenever $\sigma\in \Bisloc(\GG, \HH)$.
\end{proof}

\begin{lemma}\label{lem:ActionByPfaffianSymmetriesAndMorphisms}
	Let $\Phi\:(\GG_1, \HH_1) \to (\GG_2, \HH_2)$ be a morphism of Pfaffian groupoids (covering the identity). If $(\GG_2, \HH_2)\curvearrowright (P, \C, \pi)$ is an action by Pfaffian symmetries, then so is the induced action $(\GG_1, \HH_1)\curvearrowright (P, \C, \pi)$.  
\end{lemma}

\begin{example}
	Let $(P, \C, \pi)$ be a Pfaffian fibration and $(\GG_\C,\HH^1)\toto P$ its Pfaffian groupoid of internal symmetries (Section \ref{sec:GroupoidSymmetriesOfPfaffianFibrations}). There is a partial action of $\GG_\C$ on the prolongation defined as follows: for $j^1_x\varphi\in \GG_\C$ and $h_x\in J^1_\C P$, we set
	\[
		j^1_x \varphi \cdot h_x = T_x\varphi \circ h_x \circ (T_{\varphi(x)}\pi\circ T_x\varphi \circ h_x)^{-1} 
	\]
	which is only defined when $T\pi\circ T\varphi\circ h_x$ is invertible. This defines a partial action (Definition \ref{def:PartialAction}) $\GG_\C \paction J^1_\C P$. 
	
	Inside the groupoid of internal symmetries sits the groupoid of Pfaffian symmetries $\GG_{\C, \pi}\subseteq \GG_\C$, and the partial action $\GG_\C \paction J^1_\C P$ restricts to a complete action $\GG_{\C, \pi} \curvearrowright J^1_\C P$. 
	
	\begin{proposition}\label{prop:InternalGroupoidProlongation}
		Let $(P, \C, \pi)$ be a Pfaffian fibration.
		\begin{enumerate}
			\item The partial action $(\GG_\C, \HH^1)\paction (J^1_\C P, \C^1, \pi^{(1)})$ is an action by Pfaffian symmetries.
			\item The prolongation space $P^{(1)}\subset J^1_\C P$ is preserved by the (partial action) of the reduced Pfaffian groupoid $\overline{\GG_\C}\subseteq \GG_\C$:
			\[
				\overline{\GG_\C}\paction P^{(1)},
			\]  
			which, if smooth, is an action by Pfaffian symmetries.
		\end{enumerate}
	\end{proposition} 
	
	\begin{proof}
	The proof of statement (1) is the same as that for statement (1) of Theorem \ref{thm:InternalAndPfaffianSymmetries}, because the argument only relies pointwise on the first jet of $\varphi$.

	Statement (2) is not as immediate, since it is not longer automatically true that any first jet $j^1_x \varphi$ preserving $\C$ also preserves $P^{(1)}$. Rather, preserving $P^{(1)}$ depends on the existence of a symmetric jet, which we will now explain. 
		
	First note that the Cartan form $\theta_\C\: TP\to TP/\C$ is equivariant with respect to the natural representations of $\GG_\C$ on $TP$ and $TP/\C$: if $j^1_x\varphi\in \GG_\C$, then
	\[
		\theta_\C(T_x\varphi(X))= 	\overline{T_x\varphi} \circ \theta_\C(X).
	\]
	Recall that the curvature $\kappa_\C$ is given by $\kappa_\C(X, Y) = \theta_\C([\tilde{X}, \tilde{Y}]_x)$, where $\tilde{X}, \tilde{Y}\in \Gamma(\C)$ are any extension through $X, Y$. Now, a priori, the curvature depends pointwise on the first jets of $\tilde{X}$ and $\tilde{Y}$ in $J^1\C$. There is an action of $J^1\GG_\C$ on $J^1\C$, but the Lie bracket is only preserved by $J^2\Diffloc(P)$. Hence the curvature is only preserved by 
	\[
	\GG^{(1)}_\C = J^1\GG_\C \cap J^2\Diffloc(P).
	\]
	This shows that for $j^1_x\varphi \in \GG_\C$, the identity
	\[
		\kappa_\C(T_x\varphi(X), T_x\varphi(Y)) 	=T_x\varphi \circ \kappa_{\C}(X, Y)
	\]
	holds only when it is covered by $j^2_x\varphi \in \GG^{(1)}_\C$. It follows that $\overline{\GG_\C} = p_1(\GG^{(1)}_\C)$ preserves $P^{(1)}$.  
	\end{proof}
\end{example}
\begin{remark}
	Item (2) in Proposition \ref{prop:InternalGroupoidProlongation} seems to suggest that there might be a specific relation between the intrinsic torsion of $(\GG_\C, \HH^1)$ and of $(P, \C, \pi)$. This is true for Pfaffian actions \cite{AccorneroCattafiCrainicSalazarBook}, but we don't know yet what it is for general actions by Pfaffian symmetries.
\end{remark}

\begin{lemma}\label{lem:MorphismIntoGroupoidOfPfaffianSymmetries}
	Let $(\GG, \HH)\paction (P, \C, \pi)$ be a partial action by internal symmetries. Then there are canonical morphisms of Pfaffian groupoids
	\[
		\left((J^1_\HH \GG) \pltimes P, \pr_{J^1_\HH \GG}^*\HH^1\right) \to \left(J^1_{\pr_\GG^*\HH} (\GG \pltimes P), \HH^1\right) \to (\GG_{\C}, \HH^1).
	\]
	If $(\GG, \HH)$ acts by Pfaffian symmetries, the left map takes values in $\GG_{\C, \pi}$ and is still a morphism of Pfaffian groupoids.
\end{lemma}

\begin{proof}
	We start with the map on the left. The anchor $(m, \pr_P)\: \GG\pltimes P\to P\times P$ is a morphism of Pfaffian groupoids (with their entire tangent spaces as Pfaffian distributions). By Lemma \ref{lem:MorphismsOfPfaffianGroupoids}, it differentiates to a morphism of Pfaffian groupoids 
	\[
	m_*\: (J^1(\GG\pltimes P),\HH^1)\to (J^1\Diffloc(P), \HH^1).
	\]
	Because $(\GG, \HH)$ acts by internal symmetries, this map restricts to 
	\[
	m_*\: J^1_{\pr^*_\GG\HH}(\GG\pltimes P)\to \GG_\C,
	\]
	which is a morphism of Pfaffian groupoids. If $(\GG, \HH)$ acts by internal symmetries, the restriction of $m_*$ to $J^1_{\pr_\GG^*\HH}(\GG\pltimes P)$ takes values in $\GG_{\C, \pi}$.
	
	The map on the right is part of Lemma \ref{lem:ActionPfaffianGroupoid}.
\end{proof}

\begin{corollary}\label{cor:ActionBySymmetriesAndProlongation}
	Let $(\GG, \HH)\paction (P, \C, \pi)$ be a partial action by internal symmetries. Then:
	\begin{enumerate}
		\item The partial prolongation $J^1_\HH \GG$ acts on $J^1_\C P$ by Pfaffian symmetries.
		\item If $(\GG, \HH)$ is 2-integrable and $(P, \C, \pi)$ is 1-integrable, then $\GG^{(1)}$ acts on $P^{(1)}$ by Pfaffian symmetries.
	\end{enumerate}
\end{corollary}

\begin{proof}
	The corollary follows from combining Lemma \ref{lem:MorphismIntoGroupoidOfPfaffianSymmetries}, Lemma \ref{lem:ActionByPfaffianSymmetriesAndMorphisms}, Lemma \ref{lem:MorphismsOfPfaffianGroupoids}, Lemma \ref{lem:ActionPfaffianGroupoid} and Proposition \ref{prop:InternalGroupoidProlongation}. 
\end{proof}

\subsection{Actions by symmetries and the relative algebroid}

Finally, we are in a position where we can state how groupoid actions by symmetries of Pfaffian fibrations interact with the relative algebroid.

\begin{theorem}\label{thm:ActionsByPfaffianSymmetriesInvariance}
	Let $(\GG, \HH) \curvearrowright (P, \C, \pi)$ be an action by Pfaffian symmetries. If $(\GG, \HH)$ is 1-integrable, then the action $(\GG^{(2)}, \GG^{(1)}) \curvearrowright (\pi^*TX, \onabla, \D^\C)$ is by symmetries of relative algebroids. 
	
	If the representation $\GG^{(1)} \curvearrowright \pi^*TX$ descends to $\GG \curvearrowright \pi^*TX$, then there is an induced representation by symmetries of relative algebroids $(\overline{\GG^{(1)}}, \GG) \curvearrowright (\pi^*TX, \onabla, \D^\C)$. 
\end{theorem}
\begin{proof}
	Lemma \ref{lem:MorphismIntoGroupoidOfPfaffianSymmetries} and Lemma \ref{lem:MorphismsOfPfaffianGroupoids} imply there are groupoid morphisms $(\GG^{(2)} \ltimes P, \GG^{(1)} \ltimes P) \to (\GG_{\C, \pi}^{(1)}, \GG_{\C, \pi})$, which is a morphism of prolongation pairs. By Theorem \ref{thm:PfaffianDerivationInvariant}, the derivation $\D^\C$ is invariant for the representation $(\GG^{(1)}_\C, \GG_\C)\curvearrowright (\pi^*TX, \onabla)$, so the induced representation
	\[
		(\GG^{(2)}, \GG^{(1)}) \curvearrowright (\pi^*TX, \onabla, \D^\C)
	\]
	is by symmetries of relative algebroids.
	
	If the representation descends to $\GG \curvearrowright \pi^*TX$, then the representation $\GG^{(2)} \curvearrowright \DD^1_{(\pi^*TX, \onabla)}$ descends to $\overline{\GG^{(1)}} \curvearrowright \DD^1_{(\pi^*TX, \onabla)}$ and $\D^\C$ remains invariant for this representation (but it might fail to be invariant for the entire representation $\GG^{(1)}\curvearrowright \DD^1_{(\pi^*TX, \onabla)}$ when $\overline{\GG^{(1)}}\subsetneq \GG^{(1)}$).
\end{proof}

\begin{corollary}
	Let $(\GG, \HH) \paction (P, \C, \pi)$ be a partial action by internal symmetries. Suppose that $\GG$ is 2-integrable. Then the action
	\[
		(\overline{\GG^{(2)}}\pltimes J^1_\C P, \GG^{(1)} \pltimes J^1_\C P ) \curvearrowright ((\pi^{(1)})^*TX, \onabla, \D^{\C^1})
	\]
	is a representation by symmetries of relative algebroids. If $P$ is 1-integrable, then it restricts to a representation by symmetries of relative algebroids
	\[
		(\overline{\GG^{(2)}}\pltimes P^{(1)}, \GG^{(1)} \pltimes P^{(1)} ) \curvearrowright ((\pi^{(1)})^*TX, \onabla, \D^{\C^1}).
	\]
	
\end{corollary}

\begin{remark}
	The reduced Pfaffian groupoids suggests that there are explicit relations between the intrinsic torsion and curvature of the Pfaffian groupoid and the failure of the derivation to be invariant, so we suspect that the statements can be improved. As currently stated, however, the results follow naturally from our setup, making the proofs short. 
\end{remark}

\begin{corollary}\label{cor:DerivationPreservesInvariantSections}
	If $(\GG, \HH)\curvearrowright (P, \C, \pi)$ is an action by Pfaffian symmetries by a 2-integrable Pfaffian groupoid $(\GG, \HH)\toto M$, then $\D^\C$ preserves $\GG^{(1)}$-invariant sections:
	\[
	\D^\C\: \left(\Omega^1_{(\pi^*TX, \onabla)}\right)^{\GG^{(1)}}\to \left(\Omega^1_{\pi^*TX}\right)^{\GG^{(1)}}.
	\]
\end{corollary}

\begin{remark}\label{rk:quotients}
	In practice, the action $\GG^{(1)}\curvearrowright \pi^*TX$ descends to an action $\GG\curvearrowright \pi^*TX$. This happens for example already after the first prolongation an action by internal symmetries. In that case, the derivation preserves $\GG$-invariant sections, which means it can descend to a derivation on the quotient!
\end{remark}

\subsection{Remarks on Pfaffian actions and PDEs with symmetries}\label{sec:RemarksOnPfaffianActionsAndPDEsWithSymmetries}

In this section we show that not all PDEs with symmetries are captured by a Pfaffian action in the sense of \cite{AccorneroCattafiCrainicSalazarBook, Cattafi2020}. Pfaffian actions arise naturally in the context of geometric structures. However, the example below illustrates that the study of PDEs with symmetries from a Pfaffian perspective demands a broader notion of symmetry.

\subsubsection{Pfaffian groupoids, fibrations and actions from the perspective of forms}

The notion of a Pfaffian action relies on the Pfaffian form of a Pfaffian fibration, in contrast to this article, where only the Pfaffian distribution played a role. Let us recall the definitions using Pfaffian forms.

Let $\GG\toto M$ be a Lie groupoid and $E\to M$ a representation of $\GG$. A one-form $\omega\in \Omega^1(\GG; t^*E)$ is \textbf{multiplicative} when \cite{CrainicSalazarStruchiner2015}:
\[
(m^*\omega)_{(g, h)} = \pr_1^*\omega + L_g\circ \pr_2^*\omega,
\]
for any composable pair $(g, h)\in \GG \leftindex_{s}{\times}_t \GG$. Here $L_g\: E_{s(g)}\to E_{t(g)}$ is translation, $\pr_1$ (resp. $\pr_2$) is the projection onto the left (resp. right) factor, and $m$ is the groupoid multiplication.

If $(\GG, \HH)\toto M$ is a Pfaffian groupoid in the sense of Definition \ref{def:PfaffianGroupoid}, then, after identifying $T\GG/\HH \cong t^*(T\GG/\HH\vert_M)$, the one-form $\omega_\HH\in \Omega^1(\GG;t^*E)$, with $E=T\GG/\HH \vert_M$, is multiplicative.

In this section, a \textbf{Pfaffian groupoid} $(\GG, \omega)\toto M$ is a Lie groupoid with multiplicative one-form $\omega\in \Omega^1(\GG;t^*E)$ for which $(\GG, \ker \omega)\toto M$ is a Pfaffian groupoid in the sense of Definition \ref{def:PfaffianGroupoid}.

Similarly, by a \textbf{Pfaffian fibration} $(P, \theta, \pi)$ we mean a submersion $\pi\: P\to M$ and a one-form $\theta\in \Omega^1(P;E)$ for which $(P, \ker \theta, \pi)$ is a Pfaffian fibration in the sense of Definition \ref{def:PfaffianFibration}.

\begin{example}\label{ex:DiffeomorphismPfaffianGroupoidForm}
	If $Q$ is a manifold, we denote the Cartan form on $J^1\Diffloc(Q)$ by $\omega_Q$, and it is given by
	\[
	(\omega_Q)_{j^1_x\varphi)}(W) = Tt(W) - 	T_x\varphi \circ Ts(W)\in T_{\varphi(x)} Q.
	\]
	If $\GG\subset J^1\Diffloc(Q)$ is a subgroupoid, we obtain a Pfaffian groupoid $(\GG, \omega_Q)\toto Q$ by restriction. 
\end{example}

\begin{example}\label{ex:FirstJetPfaffianForm}
	If $q\: Q\to X$ is a submersion, the first jet $\pi\:J^1q\to X$ becomes a Pfaffian fibration when equipped with the Cartan form $\theta_Q\in \Omega^1(J^1q; p_1^*\ker Tq)$ given by
	\[
	(\theta_Q)_{j^1_x\sigma} (V)= Tp_1(V) - T_x\sigma \circ T\pi(V).
	\]
\end{example}

\begin{definition}[\cite{AccorneroCattafiCrainicSalazarBook, Cattafi2020}]\label{def:PfaffianAction}
	A groupoid action of a Pfaffian groupoid $(\GG, \omega)\toto M$ on a Pfaffian fibration $(P, \theta, \pi)$ along $\mu\: P\to M$---with $\omega\in \Omega^1(\GG;t^*E)$ and $\theta\in \Omega^1(P;\mu^*E)$---is called \textbf{Pfaffian} when 
	\begin{equation}\label{eq:PfaffianAction}
		 (\pr_\GG^*\omega)_{(g, p)} = (m^*\theta)_{(g, p)} - g\cdot(\pr_P^*\theta)_{(g, p)}
	\end{equation}
	for any $(g, p)\in \GG\ltimes P$.
\end{definition}

\begin{remark}
	Equation (\ref{eq:PfaffianAction}) implies that
	\[
	Tm(\ker \omega \leftindex_{Ts}{\times}_{T\mu} \ker \theta) \subseteq \ker \theta,
	\]
	so any Pfaffian action is also an action by internal symmetries. 
\end{remark}

\subsubsection{A PDE with symmetry that is not a Pfaffian action}

Let $q\: Q\to X$ be a submersion. The \textbf{point symmetries} of $J^1q$ come from the pseudogroup $\Diffloc(q)$, generated by local diffeomorphism $\varphi\in \Diffloc(Q)$ lifting some diffeomorphism $\overline{\varphi}\in \Diffloc(X)$.
The pseudogroup $\Diffloc(q)$ is the same as the pseudogroup of diffeomorphisms preserving $\ker Tq$. There is also a natural groupoid morphism 
\[
\overline{q}\: J^1\Diffloc(q)\to J^1\Diffloc(X),
\]
sending $j^1_x\varphi$ to $j^1_{q(x)}\overline{\varphi}$, where $\varphi$ lifts $\overline{\varphi}$ in a neighborhood of $x\in Q$. Note that $\overline{q}\: J^1\Diffloc(q)\to J^1\Diffloc(X)$ is a groupoid morphism covering $q\: Q\to X$ and the Pfaffian forms satisfy $\overline{q}^*\omega_X = Tq\circ \omega_Q$. Here $\omega_Q\in \Omega^1(J^1\Diffloc(q); t^*TQ)$ and $\omega_X\in \Omega^1(J^1\Diffloc(X);t^*TX)$ are the Cartan forms (Example \ref{ex:DiffeomorphismPfaffianGroupoidForm}).

Clearly, $\Diffloc(q)$ acts on germs of sections of $q$:
\[
\varphi\cdot \sigma = \varphi \circ \sigma \circ \overline{\varphi}^{-1}
\]
where $\varphi\in \Diffloc(q)$ lifts $\overline{\varphi}\in \Diffloc(X)$. This action descends to the level of jets and gives rise to a groupoid action
\[
\xymatrix{
	(J^1\Diffloc(q), \omega_Q) \ar@<0.5ex>[d] \ar@<-0.5ex>[d] & {(J^1q, \theta_Q, \pi)} \ar@`{[l]+/l-2.5pc/+/d+1.5pc/,[l]+/l-2.5pc/+/u+1.5pc/}[0, 0]_{m} \ar[ld]^{p_1}   \\
	Q & 
}
\]
that is neither principal nor Pfaffian. Even though the coefficients of $\omega_Q$ and $\theta_Q$ can be matched through the inclusion $\ker Tq\subset TQ$, Equation (\ref{eq:PfaffianAction}) can not hold for these forms because the RHS would be surjective onto $TQ$ while the LHS only has image contained in $\ker Tq$. Nevertheless, there is still an elegant identity relating all the forms, stated in Proposition \ref{prop:BeautifulIdentity} below, which almost looks like a ``twisted Pfaffian action".

We regard $J^1q$ as the space of pointwise splittings of $Tq$:
\[
J^1q=\{ h\: T_{q(x)}X\to T_xQ \ | \ T_xq \circ h = \id\}.
\]
\begin{proposition}\label{prop:BeautifulIdentity}
	At each point $(g, h)\in J^1\Diffloc(q)\ltimes J^1q$, the following identity holds:
	\begin{equation}\label{eq:PDEWithSymmetryForms}
		\pr^*_\GG\omega_Q - (g\cdot h) \circ \left(\overline{q}^*\omega_X\right) = m^*\theta_Q - g \cdot \left(\pr_{J^1q}^*\theta_Q\right) .
	\end{equation}
\end{proposition}
\begin{remark}
	Equation (\ref{eq:PfaffianAction}) states that the form of the Pfaffian groupoid becomes simplicially exact on the action groupoid. Equation (\ref{eq:PDEWithSymmetryForms}) can be interpreted similar, except that the form $\pr_\GG\omega_Q^*$ has to be ``twisted" by the term $(g\cdot h)\circ (\overline{q}^*\omega_X)$ before it becomes simplicially exact.
\end{remark}
\begin{proof}
	Recall the definition of the forms $\omega_Q$ and $\theta_Q$ from Examples \ref{ex:DiffeomorphismPfaffianGroupoidForm} and \ref{ex:FirstJetPfaffianForm}. If $\varphi$ descends to $\overline{\varphi}\in \Diffloc(X)$ in a neighborhood of $x\in Q$, then we can first compute the LHS of equation (\ref{eq:PDEWithSymmetryForms}):
	\begin{align*}
		(m^*\theta_Q)_{(g, h)}(W, V) &= T\mu\left( Tm(W, V)\right) - (g\cdot h) \circ T\pi\left( Tm(W, V)\right),\\
		&= T(\mu \circ m)(W, V) - T\varphi \circ h \circ T\overline{\varphi}^{-1}\left(T(\pi \circ m)(W, V)\right),\\
		&=Tt(W) - T\varphi \circ h \circ T\overline{\varphi}^{-1}\left(T(t\circ \overline{q})(W)\right),\\
		g\cdot (\pr^*_{J^1q}\theta_Q)_{(g, h)}(W, V)&= T\varphi \circ T\mu(V) - T\varphi \circ h \circ T\pi (V) \\
		&= T\varphi \circ Ts(W) -  T\varphi \circ h \circ Tq \circ Ts (W).
	\end{align*}
	where in the last equality we used $Ts(W) = T\mu(V)$. Taking their difference, we find:
	\begin{align*}
		(m^*\theta_Q - g\cdot(\pr_{J^1q}^*\theta_Q))(W, V) &= Tt(W) - T\varphi \circ Ts(W)\\
		&\phantom{=}-T\varphi \circ h \circ T\overline{\varphi}^{-1} \circ \left(Tt - T\overline{\varphi}\circ Ts\right) \left( T\overline{q}(W)\right)\\
		&= \pr_\GG^*\omega_Q(W, V) - (g\cdot h) \circ \pr_\GG^*\overline{q}^*\omega_X(W, V), 
	\end{align*}
	which proves the identity.
\end{proof}

\begin{remark}\label{rk:LucasRemark}
	As pointed out to the author by Luca Accornero, Equation (\ref{eq:PDEWithSymmetryForms}) still implies that
	\[
	Tm(\ker \omega_Q \leftindex_{Ts}{\times}_{Tp_1} \ker \theta_Q) \subseteq \ker \theta_Q,
	\]
	and it is also clear that the action preserves $\ker Tp_1$, so the action is an action by Pfaffian symmetries. This remark as been the main inspiration for developing the theory of symmetries of Pfaffian fibrations presented in this article.
\end{remark}

\begin{remark}
	The kernel of the groupoid map $\overline{q}\: J^1\Diffloc(q)\to J^1\Diffloc(X)$ is $J^1(Q\times_q Q)$, the first jet of the submersion groupoid for $q$. The action restricted to $J^1(Q\times_q Q)$ does become Pfaffian in the sense of Definition \ref{def:PfaffianAction}
\end{remark}

\appendix

\section{Relative algebroids}\label{sec:RelativeAlgebroids}

In this section, we give a rapid introduction to relative derivations and relative algebroids necessary to state our results. For more details, we refer to \cite{FernandesSmilde2025, Smilde2025Notes}. Only Sections \ref{subsec:DerivationsAsMapsOnFirstJets}, \ref{sec:RepresentationOnFlatFoliatedBundles} and \ref{sec:RepresentationOnRelativeAlgebroids} can not be found in the main sources.

\subsection{Derivations of a vector bundle}
Let $V\to N$ be a vector bundle, and let $\Omega^\bullet(V) := \Gamma(\wedge^\bullet V^*)$. Recall that a $k$-derivation $\D\in \Der^k(V)$ on $V$ is a linear map 
\[
\D \: \Omega^\bullet(V)\to \Omega^{\bullet + k}(V)
\]
satisfying a Leibniz rule:
\[
\D(\alpha\wedge\beta) = (\D \alpha)\wedge \beta + (-1)^{|\alpha| k }\alpha\wedge(\D \beta).
\]
Every derivation $\D\in \Der(V)$ has a symbol $\sigma(\D)\: \wedge^kV\to TN$ determined by
\[
\left \langle\D f, v_1\wedge \dots\wedge v_k\right\rangle = \left\langle\d f, \sigma(\D)(v_1, \dots, v_k)\right\rangle
\]
for $f\in C^\infty(N)$ and $v_1, \dots, v_k\in \Gamma(V)$. 

The space of derivations $\Der^k(V)$ is a $C^\infty(N)$-module and can be identified with the space of sections of a vector bundle $\DD^k_V\to N$. The symbol map induces the following exact sequence of vector bundles over $N$:
\[
0 \to \Hom(\wedge^{k+1}V, V)\to \DD^k_V \to \Hom(\wedge^kV, TN) \to 0.
\]
Any derivation $\D\in \Der^k(V)$ determines and is determined by a pair of a $k$-bracket $[\cdot, \dots, \cdot]\: \wedge^{k+1}\Gamma(V)\to \Gamma(V)$ and anchor $\rho\: \wedge^kV\to TN$ satisfying a Leibniz rule:
\[
[fv_0, \dots, v_k] = f[v_0, \dots, v_k] + (-1)^k \LL_{\rho(v_1, \dots, v_k)}(f)v_0
\]
for $f\in C^\infty(N)$ and $v_0, \dots, v_k\in \Gamma(V)$.

\subsection{Derivations relative to a submersion}
Let $V\to N$ be a vector bundle and $p\: M\to N$ a map (often a submersion). Denote by $W= p^*V$ and $p_*\: W\to V$ the canonical map. A \textbf{$k$-derivation on $V$ relative to $p$}, written $\D\in \Der^k(p_*)$, is a linear map
\[
\D\: \Omega^\bullet(V) \to \Omega^{\bullet + k}(W)
\]
satisfying the Leibniz rule with respect to $p$:
\[
\D(\alpha\wedge\beta) = (\D\alpha)\wedge p^*\beta+ (-1)^{|\alpha|k} (p^*\alpha)\wedge(\D \beta). 
\]
The \textbf{symbol} of a derivation $\D\in \Der^k(p)$ relative to $p$ is the bundle map $\sigma(\D) \: \wedge^k W\to p^*TN$ determined by
\[
\left\langle \D f, w_1\wedge \dots\wedge w_k\right\rangle = \left\langle \d f, \sigma(\D)(w_1, \dots, w_k)\right\rangle, 
\]
for $f\in C^\infty(N)$ and $w_1, \dots, w_k\in W$. 

This time, the space of relative derivation $\Der^k(p_*)$ is a $C^\infty(N)$-module and can be identified with the space of sections of the vector bundle $p^*\DD^k_V$. The symbol induces a short exact sequence
\[
0 \to \Hom(\wedge^{k+1}W, W) \to p^*\DD^k_V \xrightarrow{\sigma} \Hom(\wedge^k W, p^*TN)\to 0.
\]

As for ordinary derivations, a relative derivation $\D\in \Der^k(p_*)$ determines and is determined by a \textbf{$k$-bracket relative to $p$}, which consists of a bracket and anchor
\[
[\cdot, \dots, \cdot]\: \wedge^{k+1}\Gamma(V) \to \Gamma(W), \quad \rho\: \wedge^k W \to p^*TN
\]
satisfying the Leibniz rule:
\[
[fv_0, \dots, v_k] = p^*f [v_0, \dots, v_k] + (-1)^k \LL_{\rho(p^*v_1, \dots, p^*v_k)}(f)v_0,
\]
where $f\in C^\infty(N)$ and $v_0, \dots, v_k\in \Gamma(V)$. Note that the derivative makes sense because $\rho(v_1, \dots, v_k)$ takes values in $p^*TN$, so it's possible to differentiate a function on $N$ along $\rho(v_1, \dots, v_k)$, but the result will be a function on $M$.

Given an ordinary derivation $\D\in \Der^k(W)$, it is possible to restrict it only to the forms $p^*\alpha$ for $\alpha\in \Omega^\bullet(V)$. This determines a derivation relative to $p$. This process is $C^\infty(M)$-linear so there is an induced bundle map $p_*\: \DD^k_W\to p^*\DD^k_V$. This bundle map fits into a short exact sequence
\[
0\to \Hom(W, \ker Tp) \to \DD^k_W\xrightarrow{p_*} p^*\DD^k_V\to 0.
\]

\begin{example}\label{ex:PointWiseDerivations}
	Let $V\to M$ be a vector bundle and $i_x\: \{x\}\hookrightarrow N$ the inclusion of a point $x\in N$. A derivation relative to $i_x$ is simply a map
	\[
		\D\: \Omega^\bullet_V\to \wedge^{\bullet+ k}V^*_x.
	\]
	subject to the Leibniz rule. The bundle of derivations $\DD^k(i_x)\to \{x\}$ is isomorphic to $i_x^*\DD^k_V \cong (\DD^k_V)_x$, the fiber of the vector bundle $\DD^k_V$ above $x$. This is very practical way to deal with pointwise elements in $\DD^k_V$. 
\end{example}

\subsection{Flat foliated vector bundles}\label{sec:FlatFoliatedVectorBundles}
Let $W\to M$ be a vector bundle and $\FF\subset TM$ a foliation on $M$. An \textbf{$\FF$-connection} on $W$ is a partial connection
\[
\overline{\nabla}\: \Gamma(\FF)\times \Gamma(W)\to \Gamma(W)
\]
satisfying the usual properties of a connection. The \textbf{curvature} of an $\FF$-connection $\onabla$ is given by
\[
R_{\onabla}\in \Omega^2(\FF;\End(W)), \quad R_{\onabla}(X, Y)w = \onabla_X\onabla_Yw-\onabla_X\onabla_Yw-\onabla_{[X, Y]}w.
\]
The $\FF$-connection is called \textbf{flat} when $R_\onabla=0$. In this case we call $(W, \onabla)\to (M, \FF)$ a \textbf{flat foliated vector bundle}.

A local section $w\in \Gamma_W$ is called \textbf{flat} when $\onabla_X w=0$ for all $X\in \Gamma_\FF$. When $\onabla$ is flat, then locally there always exist a flat section through every $w\in W$, but nonzero global flat sections may fail to exist. We denote by $\Gamma_{(W, \onabla)}$ the sheaf of flat sections of $(W, \onabla)$.

A \textbf{map of foliations} between two foliated manifold $(M_1, \FF_1)$ and $(M_2, \FF_2)$ is a smooth map $\varphi\: M_1\to M_2$ such that $T\varphi(\FF_1)\subseteq \FF_2$.
Let $(W_1, \onabla^1)$ and $(W_2, \onabla^2)$ are two foliated vector bundles over $(M_1, \FF_1$ and $(M_2, \FF_2)$, respectively, a \textbf{map of foliated vector bundles} is a bundle map $(\Phi, \varphi)\: W_1\to W_2$ covering a map of foliations $\varphi$ such that 
\begin{equation}\label{eq:MapOfFoliatedVectorBundles}
	\Phi^*\left( \onabla^2_{\varphi_*(X)} \alpha\right) = \onabla^1_{X} \Phi^*\alpha,
\end{equation}
for all $\alpha\in \Omega^1(W_2)$ and $X\in \FF_1$. If both are foliated flat vector bundles, then Equation (\ref{eq:MapOfFoliatedVectorBundles}) implies that $\Phi^*$ preserves flat forms. If $\Phi$ is in addition fiberwise an isomorphism, then we get a map between flat sections
\[
\Phi^*\: \Gamma_{(W_2, \onabla^2)} \to \Gamma_{(W_1, \onabla^1)}, \quad (\Phi^*w_2)(x) = \Phi_x^{-1}(w_2(\varphi(x)))
\]
The converse is also true.

\begin{lemma}\label{lem:PerservingFlatSectionsImpliesMorphismOfFlatFoliatedBundles}
	Let $(W_1, \onabla^1)$ and $(W_2, \onabla_2)$ be flat foliated bundles over $(M_1, \FF_1)$ and $(M_2, \FF_2)$, respectively. If $(\Phi, \varphi)\:W_1 \to W_2$ is a fiberwise isomorphism that preserves flat sections, then $(\Phi, \varphi)$ is an isomorphism of flat foliated bundles.
\end{lemma}
\begin{proof}
	First, we show that $\varphi$ is a map of foliations. To this end, let $w_2 \in \Gamma_{(W_2, \onabla)}$, $f\in C^\infty_{\FF_2}$ and $X\in \Gamma_{\FF_1}$. Then $fw_2$ is again a flat section, so:
	\begin{align*}
		\onabla^1_X \Phi^*(fw_2) &= \LL_X (\varphi^*f)\Phi^*w_2 = \langle\d f, T\varphi(X)\rangle \Phi^*w_2 = 0.
	\end{align*}
	Since $\langle\d f, T\varphi(X)\rangle = 0$ for any basic function $f\in C^\infty_{\FF_2}$, it follows that $T\varphi(X)\in \FF_2$. 
	
	Next, we show that $\Phi$ preserves the connections. If $w_2\in\Gamma_{W_2}$ is arbitrary, we can still write $w_2 = \sum_i f_i w_2^i$ for $f_i\in C^\infty_{M_2}$ and $w_2^i\in \Gamma_{(W_2, \onabla^2)}$. Then
	\begin{align*}
		\onabla^2_{\varphi^*(X)} \Phi^*w_2 &= \sum_i \LL_{\varphi^*(X)} (\varphi^*f_i) (\Phi^*w_2^i) = \sum_i \varphi^*(\LL_X(f_i)) \Phi^*w^i_2 = \Phi^*( \onabla^2_X w_2),
	\end{align*} 
	so $\Phi$ is indeed an isomorphism of flat foliated bundles. 
\end{proof}

\subsection{Derivations relative to foliations}\label{sec:DerivationsRelativeToFoliations} The derivations associated to a Pfaffian fibration are only locally derivations relative to a submersion and their global geometry can best be described by the notion of derivation relative to a foliation.
\begin{definition}
	Let $(W, \onabla)\to (M, \FF)$ be a flat foliated vector bundle. A \textbf{$k$-derivation relative to $\FF$}, denoted $\D\in \Der^k(W, \onabla)$, is a linear map (of sheaves)
	\[
	\D\: \Omega^\bullet_{(W, \onabla)}\to \Omega^{\bullet + k}_W
	\]
	satisfying the Leibniz rule
	\[
	\D(\alpha\wedge \beta) = (\D\alpha)\wedge\beta + (-1)^{|\alpha|k} \alpha\wedge(\D \beta).
	\]
	
	The \textbf{symbol} of a derivation $\D \in \Der^k(W, \onabla)$ is the linear map $\sigma\: \wedge^k W\to \nu(\FF)$ determined by
	\[
	\left\langle \D f, w_1\wedge\dots\wedge w_k\right\rangle = \left \langle \d f, \sigma(\D)(w_1, \dots, w_k)\right\rangle
	\]
	for $f\in C^\infty_\FF\cong \Omega^0_{(W, \onabla)}$ and $w_1, \dots, w_k\in W$. 
\end{definition}
The space of relative derivations $\Der^k(W, \onabla)$ can be realized as the sections of a vector bundle $\DD^k_{(W, \onabla)}\to M$ which itself is actually also a foliated flat bundle: the flat sections of $\DD^k_{(W, \onabla)}$ correspond to those derivations that preserve $\Omega^\bullet_{(W, \onabla)}$. The symbol map induces a short exact sequence of flat foliated bundles
\[
	0 \to (\Hom(\wedge^{k+1}W, W), \onabla) \to (\DD^k_{(W, \onabla)}, \onabla) \xrightarrow{\sigma}  (\Hom(\wedge^kW, \nu(\FF)), \onabla) \to 0.
\]

As before, a derivation relative to a foliation determines and is determined by a \textbf{$k$-bracket relative to $\FF$}, which consists of a partial bracket and anchor
\[
[\cdot, \dots, \cdot]\: \wedge^{k+1} \Gamma_{(W, \onabla)}\to \Gamma_W, \quad \rho\: \wedge^kW\to \nu(\FF),
\]
satisfying the Leibniz rule:
\[
[fw_0, \dots, w_k] = f[w_0, \dots, w_k] + (-1)^k \LL_{\rho(w_1, \dots, w_k)}(f)v_0.
\]
Note that differentiating $f\in C^\infty_\FF$ in the direction of $\nu(\FF)$ makes sense because $f$ is basic for the foliation $\FF$.

An ordinary derivation $\D\in \Der^k(W)$ can be restricted to $\Omega^\bullet_{(W, \onabla)}$ to yield a derivation relative to the foliation. This process induces a bundle map $\Pi\: \DD^k_W\to \DD^k_{(W, \onabla)}$ that fits into the following short exact sequence
\[
0 \to \Hom(\wedge^k W, \FF)\to \DD^k_W \xrightarrow{\Pi} \DD^k_{(W, \onabla)} \to 0.
\]

\subsection{Derivations as maps on first jets}\label{subsec:DerivationsAsMapsOnFirstJets}
If $W\to M$ is a vector bundle, a derivation $\D \in \Der^k(W)$ factors (pointwise) through the first jet, and is therefore determined by a bundle map:
\[
\D\: J^1W^*\to \wedge^{1+k} W^*.
\]  
The investigation of its properties with respect to the structure of $J^1W^*$ is interesting but omitted from this paper. Likewise, a pointwise derivation $\D_x\in \DD^k_W$ is now a linear map 
\[
\D_x\: (J^1W^*)_x\to \wedge^{\bullet+k}W_x^*.
\]

If $(W, \onabla)\to (M, \FF)$ is a flat foliated vector bundle, then the first order PDE defining flat sections can be regarded as a subspace
\[
J^1_\onabla W^*\subset J^1W^*.
\]
A derivation relative to $\FF$ induces a bundle map
\[
\D\: J^1_{\onabla} W^*\to \wedge^{k+1}W^*. 
\]
and a pointwise derivation $\D_x\in (\DD^k_{(W, \onabla)})_x$ induces a linear map
\[
\D_x\: (J^1_{\onabla}W^*)_x\to \wedge^{k+1}W_x^*.
\]
This perspective is useful when dealing with actions of jets of bundle maps on bundles of derivations. 

If $(\Phi, \varphi)\:W\to W$ is a bundle map (which is a local isomorphism), then its first jet acts on a pointwise derivation $\D_x\in \DD^1_W$ as
\[
j^1_x\Phi\cdot \D_x = \Phi_*\circ \D_x \circ j^1_{\varphi(x)}\Phi^*\: (J^1W^*)_{\varphi(x)}\to \wedge^{k+1} W_{\varphi(x)}^*.
\]
where $j^1_{\varphi(x)}\Phi^*\: (J^1W^*)_{\varphi(x)}\to (J^1W^*)_x$ is defined as
\[
j^1_{\varphi(x)}\Phi^*\cdot j^1_{\varphi(x)}\alpha = j^1_x(\Phi^*\alpha).
\]
If, at $x\in M$, the action of $j^1_{\varphi(x)}\Phi^*$ preserves $J^1_\onabla W^*$, then the jet also acts on a pointwise derivation $\D_x\in \DD^1_{(W, \onabla)}$ by the same formula, but $\Phi$ itself does not have to be a map of flat foliated bundles away from $x$. This is the true advantage of working with jets over local sections.

Recall that derivations $\D\:\Omega^\bullet_{(W, \onabla)}\to \Omega^{\bullet+k}$ correspond to relative $k$-brackets $[\cdot, \dots, \cdot]\: \wedge^{k+1}\Gamma_{(W, \onabla)} \to \Gamma_W$. The bracket also depends only on first jets, so a pointwise derivation $\D_x\in (\DD^1_{(W, \onabla)})_x$ corresponds to a pointwise bracket
\[
[\cdot, \dots, \cdot]_x\: \wedge^{k+1}(J^1_{\onabla}W)_x \to W_x
\]  
If $j^1_{x}\Phi$ preserves $J^1_{\onabla}W$, then it acts on the partial bracket as 
\[
j^1_{x}\Phi\cdot[\cdot, \dots, \cdot]_x =\Phi_x\left( \left[(j^1_x\Phi)^{-1}(\cdot),\dots, (j^1_x\Phi)^{-1}(\cdot)\right]_x \right)\: \wedge^{k+1}(J^1_{\onabla}W)_{\varphi(x)} \to W_{\varphi(x)}. 
\]

\subsection{Relative algebroids}

A relative algebroid is nothing more than a flat foliated vector bundle together with a relative 1-derivation, and we now have the language to understand all of its ingredients.

\begin{definition}
	A relative (almost) algebroid consists of a flat foliated vector bundle $(B, \onabla)\to (M, \FF)$ together with a derivation $\D\in \Gamma(\DD^1_{(B, \onabla)})$ relative to $\FF$. We denote the structure by $(B, \onabla, \D)$. 
\end{definition}

\begin{remark}
	Since there are no conditions (yet) on the derivation of a relative algebroid, it is what should be called an almost relative algebroid, but we decided that this name is too long. Later, when we impose the proper integrability conditions on the derivation (i.e. formal integrability), we call the resulting objective a relative \emph{Lie} algebroid, where ``Lie" signifies that the formal integrability conditions are satisfied.
\end{remark}

\begin{remark}
	In practice, the foliation often arises as the fibers of a submersion $p\: M\to N$, and the flat foliated bundle $B= p^*A$ for some vector bundle $A\to N$, with the canonical flat $\ker Tp$-connection. In this case, the algebroid $(p^*A, \onabla, \D)$ relative to $\ker Tp$ is called an \textbf{algebroid relative to a submersion}, and we denote the structure by $(A, p, \D)$, where $A\to N$ is the vector bundle, $p\: M\to N$ the submersion and $\D\in \Gamma(p^*\DD^1_A)$ the relative derivation. 
\end{remark}

\begin{example}
	A derivation relative to the foliation $\FF=0$ is just an ordinary derivation. This way, any (almost) Lie algebroid, i.e. a vector bundle $A\to M$ with derivation $\D\in \Der^1(A)$ such that $\D^2=0$, can be regarded as a relative algebroid. In particular, when $R$ is a manifold, $(TR, \d)$ is a relative algebroid. 
\end{example}

Next, we introduce the concept of ``solutions" in the context of relative algebroids. 

\begin{definition}
	A \textbf{realization} $(R, \theta, r)$ of a relative algebroid $(B, \onabla, \D)$ over $(M, \FF)$ is consists of a manifold $R$ and a bundle map $(\theta, r)\: TR\to B$ (covering $r\: R\to M$) that is fiberwise an isomorphism and satisfies $\d \circ \theta^* = \theta^* \circ \D$ on $\Omega^\bullet_{(B, \onabla)}$.
\end{definition}

A realization is a particular example of a morphism of relative algebroids (Definition \ref{def:MorphismOfRelativeAlgebroids}).

\begin{definition}
	The \textbf{tableau map} of a relative algebroid $(B, \onabla, \D)$ is the map
	\[
		\tau\: \FF\to \DD^1_{(B, \onabla)}, \quad \tau(X) = \onabla_X \D.
	\]
	Here, $\onabla$ is the canonical flat $\FF$-connection on $\DD^1_{(B, \onabla)}$.
\end{definition}
The tableau map of a relative algebroid is a \textbf{tableau of derivations} \cite[Section 1]{FernandesSmilde2025}. The symbol of the tableau map defines a classical tableau map. In terms of the anchor $\rho\: B\to \nu(\FF)$, it is given by
\begin{equation}\label{eq:SymbolTableauMap}
	\sigma(\tau)(X)(b) = (\onabla_X \rho)(b) = \onabla_X (\rho(b)) - \rho(\onabla_X b).
\end{equation}

Next, we discuss two fundamental constructions of relative algebroids.

\begin{example}[Universal/tautological relative algebroid]\label{ex:Universal/tautologicalRelativeAlgebroid}
	Let $A\to N$ be a vector bundle. Recall from Example \ref{ex:PointWiseDerivations} that elements in the derivation bundle $\D_x\in \DD^1_A$ can be regarded derivations relative to the inclusion:
	\[
	\D_x\: \Omega^\bullet(A)\to \wedge^{\bullet+1}A_x^*.
	\]
	Let $p_1\: \DD^1_A\to N$ be the vector bundle projection. A derivation on $A$ relative to $p_1$ is a section of the bundle $p_1^*\DD^1_A\to \DD^1_A$. This bundle has a canonical section $\UD\in \Gamma(p_1^*\DD^1_A)$. Concretely, as a derivation, it is given by
	\[
	(\UD \alpha)\big\vert_{\D_x} = \D_x \alpha \in \wedge^{|\alpha|+1} A_x^*
	\]
	for $\alpha\in \Omega^\bullet(A)$. The relative algebroid $(A, p_1, \UD)$ is called the \textbf{universal relative algebroid} of the vector bundle $A\to N$.
	
	Geometrically, the realizations are manifolds $R$ equipped with an \emph{$A$-coframe}, which is a bundle map $(\theta, r)\: TR\to A$ that is fiberwise an isomorphism. In a way, the universal algebroid represents the classification problem for $A$-coframes without any restriction.
	
	If $A=V\to \{*\}$ is a vector space, we can identify $\DD^1_V \cong \Hom(\wedge^2V, V)$ and the bracket dual to $\UD$ is the tautological relative bracket:
	\[
		[\cdot, \cdot]^{\Breve{}}\: \wedge^2V\to C^\infty(\Hom(\wedge^2V, V), V), \quad [v, w](c) = c(v, w).
	\]
\end{example}

\begin{example}[Restriction]\label{ex:restriction}
	Often, in classification problems for geometric structures can be rephrased in terms of coframes with additional restrictions. In this example we discuss how relative algebroids arises from others by adding constraints. 
	
	Let $(B, \onabla, \D)$ be a relative algebroid over $(M, \FF)$. A map $i\: Q\to M$ is \textbf{invariant} for $(B, \onabla, \D)$ if it contains the image of the symbol (i.e. anchor). More precisely, this means that
	\[
	\im \sigma(\D_{i(x)})\subset \im \Pi\circ T_x i, \quad \mbox{for all $x\in Q$}, 
	\]
	where $\Pi\: TM\to \nu(\FF)$ is the projection onto the normal bundle.
	
	Relative algebroids can be pulled back along invariant maps under relatively mild conditions.
	
	\begin{proposition}[\cite{FernandesSmilde2025}, Proposition 4.12]\label{prop:restriction} 
		Suppose that the map $i\: Q\to M$ satisfies
		\begin{enumerate}
			\item $i$ is invariant for $(B, \onabla, \D)$,
			\item $(T_xi)^{-1}(\FF_{i(x)})$ has constant rank for all $x\in Q$.
		\end{enumerate}
		Then there is a unique structure of a relative algebroid $(i^*B, i^*\onabla, \D_Q)$ over $(Q, i^!\FF)$ for which the canonical map $i^*B\to B$ is a morphism of relative algebroids.
	\end{proposition}
\end{example}

\subsection{Prolongation and formal integrability}
For this section, let us fix a relative algebroid $(B, \onabla, \D)$ over $(M, \FF)$. A prolongation of $B$ is an attempt to complete the derivation $\D$ in such a way that ``$\D^2=0$", but, instead of considering one specific completion, we consider all of them at once. 

Recall from Section \ref{sec:DerivationsRelativeToFoliations} that there is a map $\Pi\: \DD^1_B\to \DD^1_{(B, \onabla)}$ obtained by restricting derivations to flat forms. The \textbf{space of pointwise extensions} (or the \textbf{partial prolongation space}) of $(B, \onabla, \D)$ is the space
\[
J^1_\D M := \Pi^{-1}(\im \D) = \left\{ \tilde{\D}_x\in \DD^1_{B}\ | \ \Pi(\tilde{\D}_x) = \D_x\in \DD^1_{(B, \onabla)} \right\}.
\]
It is an affine bundle over $M$ modeled on $\Hom(B, \FF)$. 

Since $p_1\: J^1_\D M\to M$ is a surjective submersion, the inclusion $i\: J^1_\D M\hookrightarrow \DD^1_{B}$ satisfies the conditions of Proposition \ref{prop:restriction} and thus the universal relative algebroid $(B, p_1, \UD))$ of $B\to M$ from Example \ref{ex:Universal/tautologicalRelativeAlgebroid} can be restricted to $J^1_\D M$ to obtain a relative algebroid that we call the \textbf{partial prolongation} of $(B, \onabla, \D)$.

However, to complete $(B, \onabla, \D)$ to an object that is closer to satisfying ``$\D^2 =0$" we have to take that condition into account. The \textbf{space of pointwise completions} (or the \textbf{prolongation space}) of $(B, \onabla, \D)$ is
\[
M^{(1)} = \{ \tilde{\D}_x\in \DD^1_B \ 
| \ \Pi(\tilde{\D}_x) = \D_x, \ \tilde{\D}_x\circ \D =0 \}
\]
The space $M^{(1)}$ is not guaranteed to be smooth, nor does $p_1\: M^{(1)}\to M$ have to be surjective. In fact, the image $p_1(M^{(1)})\subseteq M$ is precisely the locus where the intrinsic torsion of $(B, \onabla, \D)$ vanishes. We don't discuss torsion and curvature in this paper and refer to \cite{FernandesSmilde2025} for more details.

When $M^{(1)}$ is smooth and $p_1\: M^{(1)}\to M$ is surjective, we call $(B, \onabla, \D)$ \textbf{1-integrable}. In that case, the inclusion $M^{(1)}\hookrightarrow \DD^1_B$ satisfies the conditions of Proposition \ref{prop:restriction} and thus the universal relative algebroid restricts to $M^{(1)}$ (Proposition \ref{prop:restriction}). The resulting relative algebroid, denoted $(B, p_1, \D^{(1)})$ is called the \textbf{first prolongation} of $(B, \onabla, \D)$. 

Higher prolongations are defined iteratively, with
\[
(B, \onabla, \D)^{(k)} :=(B^{(k-1)}, p_k, \D^{(k)}) = (B^{(k-2)}, p_{k-1}, \D^{(k-1)} )^{(1)}. 
\]
If the first $k$ prolongations exist, we call $(B, \onabla, \D)$ \textbf{$k$-integrable}, and $(B, \onabla, \D)$ is \textbf{formally integrable} if it is $k$-integrable for all $k$. In that case, it gives rise to what should be a profinite Lie algebroid (of finite type):
\[
\xymatrix{
	{\left(B^{(\infty)}, \D^{(\infty)}\right)\:\ldots} \ar[r] \ar@<-12pt>[d] & B^{(k)} \ar[d] \ar[r] & B^{(k-1)} \ar[r] \ar[d] \ar@{-->}@/^1pc/[l]^{\D^{(k)}} & \ldots \ar[r] & B^{(1)} \ar[r] \ar[d] & B \ar[d] \ar@{-->}@/^1pc/[l]^{\D^{(1)}}\\
	M^{(\infty)}\:\ldots \ar[r]                                            & M^{(k)} \ar[r]_{p_k}    & M^{(k-1)} \ar[r]                                                               & \ldots \ar[r] & M^{(1)} \ar[r]_{p_1}    & M}               
\]
If $(B, \onabla, \D)$ is formally integrable, we say that it is a \textbf{relative Lie algebroid}, placing the ``Lie" back into the terminology.

\subsection{Morphisms of relative algebroids}

Recall from Section \ref{sec:FlatFoliatedVectorBundles} that a morphism of flat foliated vector bundles $(\Phi, \varphi)\: (B_1, \onabla^1)\to (B_2, \onabla^2)$ induces a map
\[
\Phi^*\: \Omega^\bullet_{(B_2, \onabla^2)} \to \Omega^\bullet_{(B_1, \onabla^1)}.
\]

\begin{definition}\label{def:MorphismOfRelativeAlgebroids}
	A \textbf{morphism of relative algebroids} $(\Phi, \varphi)\: (B_1, \onabla^1, \D_1) \to (B_2, \onabla^2, \D_2)$ is a map of flat foliated vector bundles such that
	\begin{equation}\label{eq:MorphismOfRelativeAlgebroids}
	\Phi^*\circ \D_2 = \D_1\circ \Phi^*, \quad \mbox{on $\Omega^\bullet_{(B_2, \onabla^2)}$.}
	\end{equation}
\end{definition}
If the morphism $\Phi$ is fiberwise an isomorphism, then Equation (\ref{eq:MorphismOfRelativeAlgebroids}) can be expressed in terms of brackets as
\[
\Phi^*[b_1, b_2]_{\D_2} = [\Phi^*b_1, \Phi^*b_2]_{\D_1},
\]
for all $b_1, b_2\in \Gamma_{(B_2, \onabla^2)}$.

This notion of morphism pairs well with the constructions built out of a relative algebroid.
\begin{proposition}[\cite{FernandesSmilde2025}, Proposition 3.21]\label{prop:NaturalityOfProlongationsOfRelativealgebroids}
	Let $(\Phi, \varphi)\: (B_1, \onabla^1, \D_1)\to (B_2, \onabla^2, \D_2)$ be a morphism of relative algebroids, such that $\Phi$ is fiberwise an isomorphism. 
	\begin{enumerate}
		\item If $(R, \theta, r)$ is a realization of $(B_1, \onabla^1, \D_1)$, then $(R, \Phi\circ \theta, \varphi\circ r)$ is a realization of $(B_2, \onabla^2, \D_2)$.
		\item If $(B_1, \onabla^1, \D_1)$ and $(B_2, \onabla^2, \D_2)$ are $k$-integrable, then $\Phi$ induces a morphism of relative algebroids
		\[
		(\Phi^{(k)}, \varphi^{(k)})\: (B_1^{(k-1)}, (p_1)_k, \D^{(k)}_1)\to (B_2^{(k-1)}, (p_2)_k, \D^{(k)}_2) 
		\]
		that is fiberwise an isomorphism.
		\item If $(B_1, \onabla^1, \D_1)$ and $(B_2, \onabla^2, \D_2)$ are formally integrable, then $\Phi$ induces a map of profinite Lie algebroids
		\[
		(\Phi^{(\infty)}, \varphi^{(\infty)})\: (B_1^{(\infty)}, \D_1^{(\infty)})\to (B_2^{(\infty)}, \D_2^{(\infty)})
		\]
		that is fiberwise an isomorphism.
	\end{enumerate}
\end{proposition}

\subsection{Representations on flat foliated vector bundles}\label{sec:RepresentationOnFlatFoliatedBundles}

The flat sections of a foliated bundle $(W, \onabla)$ are the solutions to the PDE $\onabla w=0$. Since $\onabla w$ only depends on the first jet of $w$, the PDE is determined by a subspace of $J^1W$ denoted by
\[
J^1_\onabla W = \{ j^1_xw\in J^1W \ | \ (\onabla w)_x = 0\}\subseteq J^1W.
\]
The torsion of this PDE can be identified with the curvature of $\onabla$, and when it vanishes, it is PDE-integrable, meaning that there exists flat sections trough every point $w\in W$. 

If $\GG\toto M$ is a groupoid with a representation on a vector bundle $W\to M$, then $J^1\GG\toto M$ has a representation on $J^1W\to M$. In general, it is too much to ask that $J^1\GG$ preserves the PDE for parallel section $J^1_{\onabla}W$. Therefore, we propose the following definition.

\begin{definition}
	A \textbf{prolongation pair} $(\GG_1, \GG)\toto M$ of groupoids consists of a groupoid $\GG\toto M$ and a subgroupoid $\GG_1\subseteq J^1\GG$. 
\end{definition}

\begin{definition}\label{def:RepresentationOnFlatFoliatedVectorBundle}
	A \textbf{representation} a prolongation pair $(\GG_1, \GG)\toto M$ \textbf{on a flat foliated vector bundle} $(W, \onabla)\to (M, \FF)$ is a representation $\GG\curvearrowright W$ for which $\GG_1$ preserves $J^1_\onabla W$. 
\end{definition}

\begin{remark}
	We don't necessarily require the map $\GG_1\to \GG$ to be surjective. This is to make the terminology around the groupoid of symmetries of a Pfaffian fibration more smooth. 
\end{remark}

Restricting the Cartan distribution $\HH^1$ on $J^1\GG$ to $\GG_1$ (Example \ref{ex:JetGroupoids}), we obtain a Pfaffian groupoid $(\GG_1, \HH^1)\to M$ that acts on $J^1_\onabla W$ by Pfaffian symmetries. Thus, essentially, a prolongation pair is a first order PDE on the bisections of a Lie groupoid, but our notation remembers the underlying groupoid and its representation on $W$.

\begin{proposition}\label{prop:RepresentationsOfFoliatedVectorBundles}
	Let $(\GG_1, \GG)\curvearrowright (W, \onabla)$ be a representation of flat foliated vector bundles.
	\begin{enumerate}
		\item The holonomic bisections of $(\GG_1, \HH^1)$ act by (local) isomorphisms of flat foliated vector bundles.
		\item The groupoid $\GG_1$ has a natural representation on $\DD^k_{(W, \onabla)}$.
	\end{enumerate}
\end{proposition}

\begin{proof}[Proof of Proposition \ref{prop:RepresentationsOfFoliatedVectorBundles}]
	A holonomic bisection $\tau\in \Bisloc(\GG_1, \HH^1)$ is a bisection $\tau\in \Bisloc(\GG_1)$ such that $\tau = j^1\sigma$ for some $\sigma\in \Bisloc(\GG)$. Let $L_\sigma\: W\dasharrow W$ be the local bundle isomorphism obtained from left translation by $\sigma$. Then, since $j^1\sigma$ lies in $\GG_1$, $L_\sigma$ preserves flat sections. By Lemma \ref{lem:PerservingFlatSectionsImpliesMorphismOfFlatFoliatedBundles}, it is a local isomorphism of flat foliated bundles.
	
	Recall from Section \ref{subsec:DerivationsAsMapsOnFirstJets} that a (pointwise) derivation $\D_x\in \DD^k_{(W, \onabla)}$ can be seen as a map $\D_x\: (J^1_\onabla W^*)_x\to \wedge^{k+1}W^*_x$. Since the representation of $\GG_1$ on $J^1W^*$ preserves $J^1_\onabla W$, we can define the action of $g\in \GG_1$ on $\D_{s(g)}$ as
	\[
	g\cdot \D_{s(g)} = L_g \circ \D_{s(g)} \circ (L_g)^{-1}\: (J^1_\onabla W^*)_{t(g)} \to \wedge^{k+1} W^*_{t(g)}.
	\] 
	where $L_g$ is left translation by $g$ on the appropriate bundles.
\end{proof}

\subsection{Representations on relative algebroids}\label{sec:RepresentationOnRelativeAlgebroids}
Let $(\GG_1, \GG)\toto M$ be a prolongation pair and $(B, \onabla, \D)$ a relative algebroid over $(M, \FF)$. A \textbf{representation} of $(\GG_1, \GG)$ on $(B, \onabla, \D)$ \textbf{by symmetries of relative algebroids} is a representation on a flat foliated bundle $(\GG_1, \GG)\curvearrowright (B, \onabla)$ for which $\D$ is a $\GG_1$-invariant section of $\DD^1_{(B, \onabla)}$.

\begin{example}
	Let $\GG\toto M$ be a Lie groupoid with Lie algebroid $(A, \D_A)$. Given a (local) bisection $\sigma\in \Bisloc(\GG)$, its conjugation map, given 
	\[
		c_\sigma(g) = \sigma(t(g)) \cdot g \cdot \sigma^{-1}(s(g))
	\]
	is a Lie groupoid morphism and thus differentiates to a Lie algebroid morphism $\operatorname{Ad}_\sigma\: A\to A$. Pointwise, it only depends on the first jet of $\sigma$, meaning there is a well-defined map
	\[
		\operatorname{Ad}_{j^1_x\sigma}\: A_x \to A_{t\circ \sigma(x)}
	\]
	which is the \textbf{adjoint representation} of $J^1\GG$ on $A$. Since every section of $A$ is flat, every prolongation pair $(J^1J^1\GG, J^1\GG)$ is a representation by symmetries of flat foliated bundles, so there is at least a representation $J^1J^1\GG\curvearrowright \DD^1_A$. However, the derivation $\D_A$ is only guaranteed to be invariant for the representation of $J^2\GG \subset J^1J^1\GG$. At one extreme, if $\GG\toto M$ is a bundle of Lie groupoids, then $\D_A$ is invariant for $J^1J^1\GG$, while at the other extreme, if $\GG= M\times M\toto M$, then an element in $J^1J^1\GG$ leaves the de Rham differential invariant on $A=TM$ invariant if and only if it lies in $J^2\GG$. 
	
	The representation $(J^2\GG, J^1\GG)\curvearrowright (A, \D_A)$ is a representation by symmetries of (relative) algebroids.
\end{example}

\begin{example}
	Recall from Example \ref{ex:Universal/tautologicalRelativeAlgebroid} the tautological relative algebroid associated to a vector space $V$. The group $\operatorname{GL}(V)$ acts on $\Hom(\wedge^2V, V)$ by $(g\cdot c)(v, w) = g\cdot c(g^{-1}v, g^{-1}w)$. 
	
	The action groupoid $\operatorname{GL}(V) \ltimes \Hom(\wedge^2V, V)\toto \Hom(\wedge^2V, V)$ has an obvious representation on the trivial bundle $\underline{V} \to \Hom(\wedge^2V, V)$. Taking first jets of constant bisection of the action groupoids, we can regard $\GL(V)\ltimes \Hom(\wedge^2V, V)$ as a subgroupoid of its own first jet.
	
	The prolongation pair $(\GL(V)\ltimes \Hom(\wedge^2V, V), \GL(V)\ltimes \Hom(\wedge^2V, V))$ acts on $(V, \Breve{p}_1, \UD)$ by symmetries of relative algebroids. Indeed, if $[\cdot, \cdot]^{\Breve{}}$ is the bracket, then
	\[
		\left(g \cdot [\cdot, \cdot]^{\Breve{}}\right)(v, w)(c) = [gv, gw]^{\Breve{}}(g\cdot c) = (g\cdot c)( gv, gw) = g(c(v, w)),
	\]
	so the bracket (and thus the dual derivation) is invariant.
\end{example}

\section{Tableau maps and involutivity}\label{app:TableauMapsAndInvolutivity}

In the classical theory of partial differential equations, a \textbf{tableau} is simply a subspace $\mf{g} \subset \Hom(V, W)$ for two (finite-dimensional) vector spaces $V$ and  $W$. In the context of Pfaffian fibrations or relative algebroids, one encounters \textbf{tableau maps}
\[
	\tau\: \mf{g}\to \Hom(V, W)
\]
that are not necessarily injective. These were called \textit{generalized tableaux} in \cite{Salazar2013}, but we stick to the term tableau maps to avoid confusion with other generalizations of tableaux (such as tableaux of derivations in \cite[Section 1]{FernandesSmilde2025}). In this section, we show that several natural extensions of definitions of involutivity for tableau maps coincide.

A tableau map corresponds intuitively to a simple system of first order linear differential equation on pairs $(f, g)$ of functions $f\: V\to W$ and $g\:V\to \mf{g}$ such that  
\[
	\d_v f  = \tau(g(v))
\]
for all $v\in \mf{g}$. In case $\mf{g} = \ker \tau \oplus \im\tau$ is split, then this is equivalent to a pair of functions $(f, h)$ with $f\: V\to W$ and $h\: V\to \ker \tau$ such that $\d f \mod \im \tau = 0$, with no further equations on $h$.

When $\tau\: \mf{g}\to \Hom(V, W)$ is a tableau map, the Spencer differential w.r.t. $\tau$ is 
\[
	\delta_\tau\: \Hom(V, \mf{g})\to \Hom(\wedge^2V, W), \quad \delta_\tau(\xi)(v, w) = \tau(\xi(v))(w) - \tau(\xi(w))(v),
\]
for $v, w\in V$. 

\begin{definition}
	The first prolongation $\tau^{(1)}$ of a tableau map $\tau$ is the subspace
	\[
		\mf{g}^{(1)} = \{ \xi\in \Hom(V, \mf{g})\ | \ \delta_\tau\xi = 0\},
	\]
	which is a tableau in the classical sense.
\end{definition}
A \textbf{flag} for an $n$-dimensional vector space $V$ is a nested sequence of subspaces $\{V_k\}$ such that $\dim V_k = k$. A basis adapted to the flag $\{V_k\}$ is a basis $\{e_k\}_k$ with the property that $\{e_1, \dots e_k\}$ is a basis of $V_k$. 

For a tableau map $\tau\: \mf{g}\to \Hom(V, W)$, we set
\[
	\mf{g}_k = \{ \xi \in \mf{g} \ | \ \tau(\xi)\vert_{V_k} = 0\}.
\]
Note that $\mf{g}_0 = \mf{g}$ and $\mf{g}_n = 0$. 

The Cartan characters of $\tau$ (w.r.t. the flag $(V_k)_k$) are defined to be
\[
	s^{\mf{g}}_k = \dim \mf{g}_{k-1} - \dim \mf{g}_k.
\]
In the proposition below, we set $\mf{h} = \im \tau\subset \Hom(V, W)$. Note that $\mf{h}$ is a tableau in the classical sense.
\begin{proposition}[Cartan's bound and Cartan's test]\label{prop:CartansBoundTableauMaps}
	Let $\tau\: \mf{g}\to \Hom(V, W)$ be a tableau map. Then
	\[
		\dim \mf{g}^{(1)} \leq \sum_{k=1}^n ks^{\mf{g}}_k.
	\]
	For a flag $(V_k)_k$ of $V$ with adapted basis $\{e_k\}_k$, the following are equivalent.
	\begin{enumerate}
		\item $\dim \mf{h}^{(1)} = \sum_k k s^{\mf{h}}_k$.
		\item $\dim \mf{g}^{(1)} = \sum_{k} ks^{\mf{g}}_k$.
		\item The map $\iota_{e_k}\: \left(\mf{h}^{(1)}\right)_{k-1} \to \mf{h}_{k-1}$ is surjective for all $k$. 
		\item The map $\iota_{e_k}\: \left(\mf{g}^{(1)}\right)_{k-1} \to \mf{g}_{k-1}$ is surjective for all $k$.
	\end{enumerate}
\end{proposition}
In other words, the tableau map $\tau$ is involutive if and only if $\im \tau$ is involutive as a classical tableau. 
\begin{definition}
	A tableau map $\tau\: \mf{g}\to \Hom(V, W)$ is \textbf{involutive} if there is a flag $(V_k)_k$ such that Cartan's bound is achieved:
	\[
		\dim \mf{g}^{(1)} = \sum_{k=1}^n k s^{\mf{g}}_k.
	\]
\end{definition}

\begin{proof}[Proof of Proposition \ref{prop:CartansBoundTableauMaps}]
	The map $\tau\: \mf{g}\to \mf{h}\subseteq \Hom(V, W)$ induces a map 
	\[
		\tau_*\: \Hom(V, \mf{g})\to \Hom(V, \mf{h}).
	\]
	If $(V_k)_k$ is a flag for $V$ with adapted basis $\{e_k\}$, then there is the following diagram
	\[
	\xymatrix{
		{} & 0 \ar[d] & 0 \ar[d] & 0\ar[d] & {} \\
		0 \ar[r]& \Hom(V/V_k, \ker \tau) \ar[r] \ar[d] & (\mf{g}^{(1)})_k \ar^-{\tau_*}[r] \ar[d] & (\mf{h}^{(1)})_k \ar[r]\ar[d] & 0 \\
		0 \ar[r]& \Hom(V/V_{k-1}, \ker \tau) \ar[r]\ar^-{\iota_{e_k}}[d] & (\mf{g}^{(1)})_{k-1} \ar^-{\tau_*}[r] \ar^-{\iota_{e_k}}[d]& (\mf{h}^{(1)})_{k-1} \ar^-{\iota_{e_k}}[d] \ar[r] & 0 \\
		0\ar[r] & \ker \tau \ar[r] \ar[d]& \mf{g}_{k-1} \ar^-{\tau}[r] \ar[d] & \mf{h}_{k-1} \ar[d] \ar[r] & 0 \\ 
		{} & 0 & \bullet \ar^-{\cong}[r] & \bullet & {}		
	}
	\]
	with exact rows. From the Snake Lemma, it follows that $\iota_{e_k}\: (\mf{g}^{(1)})_{k-1}\to \mf{g}_{k-1}$ is surjective if and only if $\iota_{e_k} \:(\mf{h}^{(1)})_{k-1}\to \mf{h}_{k-1}$ is, which shows that (3) and (4) are equivalent.
	
	Furthermore, the diagram implies that
	\[
		s^{\mf{g}}_k = \begin{cases}
			s^{\mf{h}}_k, &k<n\\
			s^{\mf{h}}_n + \dim \ker \tau, & k=n.
		\end{cases}
	\]
	Since $\dim \mf{g}^{(1)} = \dim h^{(1)} + n \dim \ker \tau$, equivalence of (1) and (2) also follows. Equivalence of (1) and (3) is standard (e.g. \cite[Proposition IV.3.6]{BryantChern1991}).
\end{proof}

\bibliographystyle{abbrv}
\bibliography{bib.bib}

\end{document}